\title{Linear regression over the max-plus semiring: \\  algorithms and applications}
\author{James Hook}

\documentclass[12pt]{article}
\usepackage{eqnarray,amssymb}
\usepackage{amsmath,amsthm}
\usepackage{pgf}
\usepackage{tikz}
\usetikzlibrary{arrows,automata}
\usepackage[latin1]{inputenc}
\usepackage{graphicx}
\newtheorem{theorem}{Theorem}[section]
\newtheorem{lemma}[theorem]{Lemma}
\newtheorem{proposition}[theorem]{Proposition}
\newtheorem{example}[theorem]{Example}

\usepackage{float}
\usepackage{enumitem}
\usepackage{subfigure}
\usepackage{caption}
\usepackage{longtable}
\usepackage{stmaryrd}

\usepackage{bm}

\DeclareFontFamily{OT1}{pzc}{}
\DeclareFontShape{OT1}{pzc}{m}{it}{<-> s * [1.10] pzcmi7t}{}
\DeclareMathAlphabet{\mathpzc}{OT1}{pzc}{m}{it}

\usepackage{algorithm}
\usepackage{algpseudocode}

\DeclareFontFamily{OT1}{pzc}{}
\DeclareFontShape{OT1}{pzc}{m}{it}{<-> s * [1.10] pzcmi7t}{}
\DeclareMathAlphabet{\mathpzc}{OT1}{pzc}{m}{it}

\def\uc#1{\mathcal{#1}}

\def\Rmax{\mathbb{R}_{\max}}
\def\Rmin{\mathbb{R}_{\min}}
\def\R{\mathbb{R}}

\def\col{\hbox{\normalfont col}}
\def\support{\hbox{\normalfont support}}

\def\pattern{\hbox{\normalfont pattern}}
\def\Cl{\hbox{\normalfont Cl}}

\def\relint{\hbox{\normalfont relint}}

\def\uc{\mathcal}

\usepackage{subfigure}
\usepackage{caption}

\newtheorem{problem}[theorem]{Problem}

\theoremstyle{definition}
\usepackage[hmargin=2cm,vmargin=2cm]{geometry}

\begin{document}
\maketitle
\begin{abstract}

In this paper we present theory, algorithms and applications for regression over the max-plus semiring. We show how max-plus $2$-norm regression can be used to 
obtain maximum likelihood estimates for three different inverse problems. Namely inferring a max-plus linear dynamical systems model from a noisy time series recording, inferring the edge lengths of a network from shortest path information and fitting a max-plus polynomial function to data.

\end{abstract}

\section{Introduction}

Max-plus algebra concerns the max-plus semiring $\Rmax=[\R\cup\{-\infty\},\oplus,\otimes]$, with
\begin{equation}
a\oplus b=\max\{a,b\}, \quad a\otimes b=a+b, \quad \hbox{for all $a,b\in\Rmax$}.
\end{equation}
A max-plus matrix is an array of elements from $\Rmax$ and max-plus matrix multiplication is defined in analogy to the classical (i.e. not max-plus) case. For $A\in\Rmax^{n\times d}$ and $B\in\Rmax^{d\times m}$ we have $A\otimes B\in\Rmax^{n\times m}$ with 
\begin{equation}
(A\otimes B)_{ij}=\max_{k=1}^{d}(a_{ij}+b_{kj}),
\end{equation}
for $i=1,\dots,n$, $j=1,\dots,m$. Max-plus algebra has found a wide range of applications in operations research, dynamical systems  and control \cite{butk10,heid2006,DESCHUTTER20011049}. In this paper we make a detailed study of the max-plus $p$-norm regression problem with a view to developing new algorithms for max-plus algebraic data analysis.
\begin{problem}\label{mppnreg}
For $A\in\Rmax^{n\times d}$,  $\bm{y}\in\Rmax^n$ and $p\geq 1$, we seek
\begin{equation}\label{red443}
\min_{\bm{x}\in\Rmax^d}\|A\otimes \bm{x}-\bm{y}\|_{2}.
\end{equation}
\end{problem}
Problem~\ref{mppnreg} has already received some attention in the $\infty$-norm case, in connection with the development of methods for solving max-plus linear systems exactly, which have applications in scheduling \cite[Chapter 3]{butk10}. In the  $\infty$-norm case it is possible to compute an optimal solution with cost $\uc{O}(nd)$. However, the  $\infty$-norm residual does not model any typical noise process and consequently this regression problem is not directly useful for solving practical inverse problems.  The $2$-norm residual models Gaussian noise and is consequently the most widely used residual in classical inverse problems. In Section~\ref{2normo} we show that the $2$-norm residual is non-smooth and non-convex, which makes it difficult to optimize. Indeed, we show further that even determining whether a point $\bm{x}\in\Rmax^d$ is a local minimum of the residual surface is an NP-hard problem. However, in spite of these apparent difficulties we find that a variant of Newton's method with undershooting is able to quickly return approximate solutions that are sufficiently close to optimal to provide good estimates for the inverse problems that we investigate in Sections~\ref{tss},\ref{netset} and~\ref{polf}.

\medskip

Virtually every application of max-plus algebra in dynamical systems and control exploit its ability to model certain classically non-linear phenomena in a linear way, as illustrated in the following example. 
\begin{example}Consider a distributed computing system in which $d$ processors iterate a map in parallel. At each stage processor $i$ must wait until it has received input from its neighboring processors before beggining its next local computation. Then after completing its local computation it must broadcast some output to its neighboring processors. Define the vectors of update times $\bm{t}(0),\dots,\bm{t}(N)\in\Rmax^d$, by $\bm{t}(n)_{i}=$ the time at which processor $i$ completes its $n$th local computation. These update times can be modeled by 
\begin{equation}\label{ffffffff}
\bm{t}(n+1)=M\otimes \bm{t}(n), \quad \hbox{for $n=0,\dots,N-1$,}
\end{equation}
where $M\in\Rmax^{d\times d}$ is the max-plus matrix given by
\begin{equation}
m_{ij}=\left\{\begin{array}{cc} a_{i}+c_{ij}, & \hbox{if $j\in J_{i}$}, \\ -\infty, & \hbox{otherwise}, \end{array}\right.
\end{equation}
where $a_{i}$ is the time taken for processor $i$'s local computation, $c_{ij}$ is the time taken for communication from processor $j$ and processor $i$ receives input from the processors $J_{i}\subset\{1,\dots,d\}$, for $i,j=1,\dots,d$. The update rule \eqref{ffffffff} constitutes a max-plus linear dynamical system. By studying the max-plus algebraic properties of the matrix $M$ we can now predict the behavior of the system, for example computing its leading eigenvalue to determine the average update rate of the computations iteration. 
\end{example}

Using petri-net models, such max-plus linear models can be derived for more complicated systems of interacting timed events  \cite[Chapter 7]{heid2006}. These linear models can be extended by introducing stochasticity, which in the above example could model random variability in the time taken for messages to pass through the computer network \cite[Chapter 11]{hook2013,heid2006}, by allowing the system to switch between one of several governing max-plus linear equations \cite{vandenBoom2012}, or by including a controller input \cite{DESCHUTTER20011049}. This approach has been used to model a wide variety of processes including the Dutch railway system \cite[Chapter 8]{heid2006}, mRNA translation \cite{BRACKLEY2012128} and the Transmission Control Protocol (TCP) \cite{Baccelli:2000:TML:347057.347548}. 

In this context \emph{forwards problems} arise by presupposing a dynamical systems model then asking questions about how its orbits must behave. Conversely an \emph{inverse problem} is to infer a dynamical systems model from an empirical time series recording. In the control theory literature this inverse problem is referred to as \emph{system identification}. For example in \cite{Farahani2014,1184996,7085024,7082376} the authors present methods for system identification of stochastic max-plus linear control systems. These methods, which can be applied to a very wide class of system, with non-Gaussian noise processes, work by formulating a non-linear programming problem for the unknown system parameters, which is then solved using one of several possible standard gradient based algorithm. However, the resulting problems are necessarily non-smooth and non-convex, which makes the optimization difficult. 

Since these optimization problems are very complicated and difficult to solve our approach is to study them in the simplest possible setting, which we take to be Problem~\ref{mppnreg}. A great deal of theory has already been developed for max-plus linear algebra and these results are also more easily utilized in this simpler setting. In Section~\ref{tss} we show how max-plus $2$-norm regression can be used to obtain maximum likelihood estimates for the inverse problem of determining a max-plus linear dynamical systems model from a noisy time series recording.  

\begin{figure}
\begin{center}

\subfigure[]{\begin{tikzpicture}[scale=0.5,-,>=stealth',shorten >=1pt,auto,node distance=1.5cm,main node/.style={font=\sffamily\footnotesize\bfseries}]

  \node[main node] (x1) at (0,0) {$x(1)$};
  \node[main node] (x2) at (0,-2) {$\cdots$};
    \node[main node] (x3) at (0,-4) {$x(n)$};

  \node[main node] (v1) at (4,-1) {$y(1)$};
  \node[main node] (vv1) at (4,-2) {$\cdots$};
 \node[main node] (v2) at (4,-3) {$y(d)$};

  \node[main node] (z1) at (8,0) {$z(1)$};
  \node[main node] (z2) at (8,-2) {$\cdots$};
    \node[main node] (z3) at (8,-4) {$z(m)$};

  \path[every node/.style={font=\sffamily\footnotesize}]

    (x1) edge [->,line width =0.25mm] node [left]  {} (v1)
    (x1) edge [->,line width =0.25mm] node [near end] {} (v2)
    
     (x2) edge [->,line width =0.25mm] node  {} (v1)
    (x2) edge [->,line width =0.25mm] node [near end] {} (v2)

 (x3) edge [->,line width =0.25mm] node  {} (v1)
    (x3) edge [->,line width =0.25mm] node [right] {} (v2)

    (z1) edge [<-,line width =0.25mm] node [left]  {} (v1)
    (z1) edge [<-,line width =0.25mm] node [near end] {} (v2)
    
     (z2) edge [<-,line width =0.25mm] node  {} (v1)
    (z2) edge [<-,line width =0.25mm] node [near end] {} (v2)

 (z3) edge [<-,line width =0.25mm] node  {} (v1)
    (z3) edge [<-,line width =0.25mm] node [right] {} (v2);

 \end{tikzpicture}}\hspace{1cm}\subfigure[]{ \begin{tikzpicture}[scale=0.5,-,>=stealth',shorten >=1pt,auto,node distance=1.5cm,main node/.style={font=\sffamily\footnotesize\bfseries}]

  \node[main node] (x1) at (0,0) {$x(1)$};
  \node[main node] (x2) at (0,-2) {$\cdots$};
    \node[main node] (x3) at (0,-4) {$x(n)$};

  \node[main node] (v1) at (4,-1) {$y(1)$};
  \node[main node] (vv1) at (4,-2) {$\cdots$};
 \node[main node] (v2) at (4,-3) {$y(d)$};

  \path[every node/.style={font=\sffamily\footnotesize}]

    (x1) edge [<->,line width =0.25mm] node [left]  {} (v1)
    (x1) edge [<->,line width =0.25mm] node [near end] {} (v2)
    
     (x2) edge [<->,line width =0.25mm] node  {} (v1)
    (x2) edge [<->,line width =0.25mm] node [near end] {} (v2)

 (x3) edge [<->,line width =0.25mm] node  {} (v1)
    (x3) edge [<->,line width =0.25mm] node [right] {} (v2);

 \end{tikzpicture}}
 \caption{Some simple network structures.}\label{smallnetfig}
\end{center}
\end{figure}
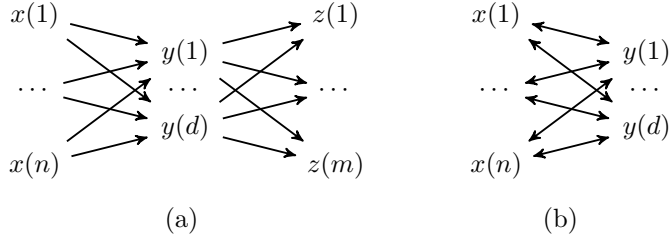

\medskip
 
Tropical algebra is the more general filed of mathematics encompassing any semiring whose `addition' operation is max or min, for example the min-plus and max-times semirings.
Min-plus algebra concerns the min-plus semiring $\Rmin=[\R\cup\{+\infty\},\boxplus,\boxtimes]$, with
\begin{equation}
a\boxplus b=\min\{a,b\}, \quad a\boxtimes b=a+b, \quad \hbox{for all $a,b\in\Rmin$}.
\end{equation}
The min-plus and max-plus semirings are isomorphic via the map $h:\Rmax\leftrightarrow \Rmin$, with $h(x)=-x$. Applied componentwise this map also preserves the $p$-norm of a vector, so that the max-plus and min-plus regression problems are mathematically equivalent. Min-plus matrix algebra naturally describes shortest paths through graphs, as illustrated in the following examples.

\begin{example}\label{integg1}Consider the network illustrated in Figure~\ref{smallnetfig} (a). We can think of the $x$ vertices as starting points, the $y$ vertices as transport hubs and the $z$ vertices as destinations. Suppose that $M_{L}\in\Rmin^{n\times d}$ and $M_{R}\in\Rmin^{d\times m}$ are min-plus matrices such that $(m_{L})_{ik}$ is the length of the edge from $x(i)$ to $y(k)$ and $(m_{R})_{kj}$ is the length of the edge from $y(k)$ to $z(j)$. Then $D=M_{L}\boxtimes M_{R}$ is the $n\times m$ min-plus matrix such that 
\begin{equation}
d_{ij}=\min_{k=1}^{d}\big((m_{L})_{ik}+(m_{R})_{kj}\big)
\end{equation}
is the length of the shortest path from $x(i)$ to $z(j)$, for $i=1,\dots,n$, $j=1,\dots,m$. 
\end{example}
\begin{example}\label{integg2}
For the network illustrated in Figure~\ref{smallnetfig} (b) suppose that $M\in\Rmin^{n\times d}$ is the min-plus matrix such that $m_{ij}$ is the length of the edge between $x(i)$ and $y(j)$. Then 
\begin{equation}
D=I\boxplus M\boxtimes M^{\top} \boxplus (M\boxtimes M^{\top})^{\boxtimes 2} \boxplus \cdots \boxplus (M\boxtimes M^{\top})^{\boxtimes( n-1)},
\end{equation}
is the $n\times n$ min-plus matrix such that $d_{ij}$ is the length of the shortest path from $x(i)$ to $x(j)$, where $I$ is the min-plus identity matrix with zeros on the diagonal an infinities off of the diagonal. 
\end{example}

Therefore min-plus matrix multiplication and addition are forwards operators that map local information about edges to global information about shortest paths. In this context an inverse problem is to infer information about the edges from possibly noisy or partial information about shortest paths. In Section~\ref{netset} we show how min-plus low-rank approximate matrix factorization can be used to obtain maximum likelihood estimates for some of these inverse problems. We show further how this process can be applied to more general network structures to provide a kind of min-plus model order reduction, that could be useful for characterizing networks or extracting useful features to characterize individual vertices in a network. The basis for our max-plus low-rank approximate matrix factorization comes from the previously developed regression algorithms.
\medskip 

Max-times algebra concerns the max-times semiring, which is the algebra of the non-negative real numbers along with the binary operations max and times. Although max-plus and max-times are isomorphic as algebraic structure, via the map $h:\R_{\max +}\mapsto \R_{\max \times}$, defined by $h(x)=\log(x)$, this isomorphism does not preserve any $p$-norm and consequently approximation in max-plus is not compatible with approximation in max-times. Max-times approximate low-rank matrix factorization has been explored as a an alternative and companion to classical non-negative matrix factorization \cite{capricorn,cancer,latitude}. Intuitively non-negative matrix factorization represents each component of a whole object as a sum of its parts, whilst in max-times factorization each component of an object is represented by a single part in a `winner takes all' regime. Although max-plus and max-times inverse problems are not isomorphic there are some clear similarities between them, most strikingly how the max operation introduces non-differentiability and results in large open patches of the residual having zero derivative with respect to certain variables.

\medskip 

The remainder of this paper is organized as follows. In Section~\ref{reg} we review what is known for the $\infty$-norm regression problem before developing some theory and algorithms for the $2$-norm case. Then in Sections~\ref{tss}~,\ref{netset} and~\ref{polf} we show how max-plus $2$-norm regression can be used to obtain maximum likelihood estimates for three different inverse problems. We also include an appendix, which contains an algorithm for exactly solving the max-plus $2$-norm regression problem, a proof of the result that determining whether a point is a local minimum is an NP-hard problem and an algorithm for computing symmetric min-plus low-rank approximate matrix factorizations.

\section{Max-plus regression}\label{reg}
The column space of a max-plus matrix $A\in\Rmax^{n\times d}$ is simply the image of the matrix vector multiplication map
\begin{equation}
\col(A)=\{A\otimes \bm{x}~:~\bm{x}\in\Rmax^d\}.
\end{equation}
Just as in the classical case, the $p$-norm regression problem can be written as an optimization over the column space of the matrix
\begin{equation}\label{colview}
\min_{\bm{x}\in\Rmax^d}\|A\otimes \bm{x}-\bm{y}\|_{p}=\min_{\bm{z}\in\col(A)}\|\bm{z}-\bm{y}\|_{p}.
\end{equation}
Understanding the geometry of the column space is therefore key to understanding the regression problem. 
\begin{example}\label{inf2egg}
Consider 
$$
A=\left[\begin{array}{cc} 0 & 0 \\ 1 & 0 \\ 0 & 1 \end{array}\right], \quad \bm{y}=\left[\begin{array}{cc} 1  \\ 1  \\ 1 \end{array}\right].
$$
The column space of $A$ is given by the union of two simplices  
$$
\col(A)=\{[x_{1},x_1+1,x_2+1]^{\top}~:~x_2 \leq x_1 \leq x_2+1\}\bigcup\{[x_2,x_1+1,x_2+1]^{\top}~:~x_1 \leq x_2 \leq x_1+1\}.
$$
Equivalently $\col(A)$ is a prism with an L-shaped cross section
$$
\col(A)=\{\bm{l}+[\alpha,\alpha,\alpha]^{\top}~:~\bm{l} \in L,~\alpha\in\Rmax\},
$$
where
$$
L=\{[0,t,0]^{\top}~:~t\in[0,1]\}\bigcup \{[0,0,t]^{\top}~:~t\in[0,1]\}.
$$
Now consider Problem~\ref{mppnreg} with $p=\infty$. Figure~\ref{pinfp2} (a) displays the column space of $A$ along with the target vector $\bm{y}$. We have also plotted the ball 
$$
B_{\infty}(\bm{y},1/2)=\{\bm{y'}\in\Rmax^3~:~\|\bm{y}'-\bm{y}\|_{\infty}=1/2\},
$$
which is the smallest such ball that intersects $\col(A)$. Therefore the minimum value of the residual is $1/2$ and the closest points in the column space are given by the L-shaped 
set 
$$
\arg\min_{\bm{z}\in\col(A)}\|\bm{z}-\bm{y}\|_{\infty}=\{[1/2,3/2,t]^{\top}~:~t\in[1/2,3/2]\}\bigcup \{[1/2,t,3/2]^{\top}~:~t\in[1/2,3/2]\}.
$$
Next consider Problem~\ref{mppnreg} with $p=2$. Figure~\ref{pinfp2} (b) displays the column space of $A$ along with the target vector $\bm{y}$. We have also plotted the ball 
$$
B_{2}(\bm{y},1/\sqrt{2})=\{\bm{y'}\in\Rmax^3~:~\|\bm{y}'-\bm{y}\|_{2}=1/\sqrt{2}\},
$$
which is the smallest such ball that intersects $\col(A)$. Therefore the minimum value of the residual is $1/\sqrt{2}$ and the closest points in the column space are given by
$$
\arg\min_{\bm{z}\in\col(A)}\|\bm{z}-\bm{y}\|_{\infty}=\{[1/2,3/2,1]^{\top},[1/2,1,3/2]^{\top}\}.
$$

\end{example}

\begin{figure}[t]
\begin{center}
\subfigure[$p=\infty.$]{\includegraphics[scale=0.225]{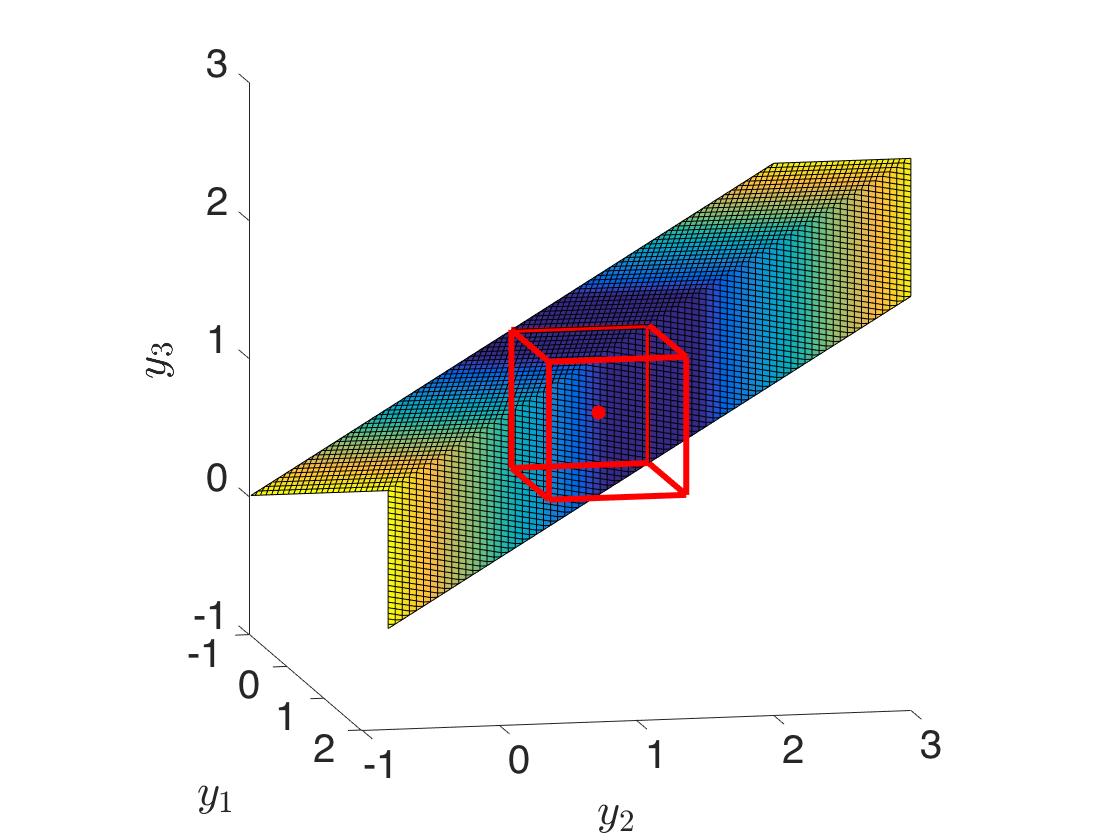}}\subfigure[$p=2.$]{\includegraphics[scale=0.225]{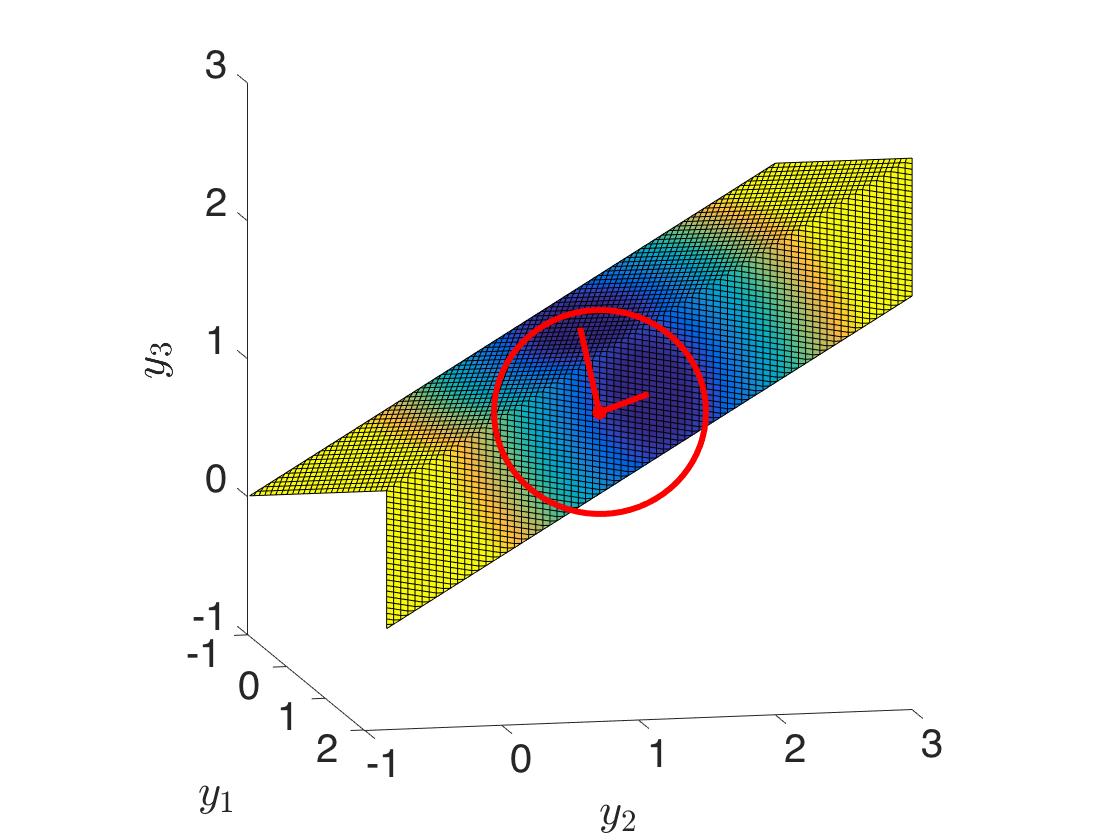}} 
\caption{Column space view of max-plus regression.}\label{pinfp2}
\end{center}
\end{figure}

\subsection{$\infty$-norm regression}

This variant of Problem~\ref{mppnreg} has been previously studied. See \cite[Section 3.5]{butk10} and the references therein. 
We saw in Example~\ref{inf2egg} that the max-plus $\infty$-norm regression problem could support multiple optimal solutions comprising a non-convex set. However we find that the max-plus $\infty$-norm regression problem is convex with respect to max-plus algebra and that we are able to very easily compute an optimal solution for it.

A function $f:\mathbb{R}_{\max}^{n}\mapsto\mathbb{R}_{\max}$ is \emph{max-plus convex} if for all $\bm{x},\bm{y}\in\mathbb{R}_{\max}^{n}$, and $\lambda,\mu\in\mathbb{R}_{\max}^{n}$ such that $\lambda\oplus \mu=0$,  we have
\begin{equation}
f(\lambda\otimes \bm{x}\oplus \lambda\otimes \bm{y})\leq \lambda\otimes f(\bm{x})\oplus \mu\otimes f(\bm{y}).
\end{equation}
Similarly a set  $X\subset \mathbb{R}_{\max}^{n}$  is \emph{max-plus convex} if $\bm{x},\bm{y}\in X$, and $\lambda,\mu\in\mathbb{R}_{\max}^{n}$ such that $\lambda\oplus \mu=0$, we have $\lambda\otimes \bm{x}\oplus \lambda\otimes \bm{y}\in X$. See e.g. \cite{Gaubert2006}. It follows that the minima of a max-plus convex function form a max-plus convex set and therefore that any local minimum is also a global minimum and that the set of all global minima form a single path connected set. The following result is straightforwards to prove.

\begin{proposition}\label{mpconvthm}
Let $A\in\Rmax^{n\times d}$ and $\bm{y}\in\Rmax^n$, then $R_{\infty}:\Rmax^{n}\mapsto \Rmax$, defined by $R_{\infty}(\bm{x})=\|A\otimes \bm{x}-\bm{y}\|_{\infty}$, is max-plus convex.
\end{proposition}

We can compute an optimal solution for the max-plus $\infty$-norm regression problem as follows. For $A\in\Rmax^{n\times n}$ and $\bm{y}\in\Rmax^n$. Let 
\begin{equation}\label{INFSOL}
\hat{\bm{x}}=-\big(A^{\top}\otimes(-\bm{y})\big), \quad \bm{x}^{\ast}=\hat{\bm{x}}\otimes \alpha/2,
\end{equation}
where $\alpha=\|A\otimes \hat{\bm{x}}-\bm{y}\|_{\infty}$. The vector $\bm{x}^{\ast}$ can be computed with cost $\uc{O}(nd)$. 

\begin{theorem}[\cite{butk10}, Theorem 3.5.2]
Let $A\in\Rmax^{n\times d}$ and $\bm{y}\in\Rmax^n$, then 
\begin{equation}
\bm{x}^{\ast}=\sup_{\leq}\big(\arg\min_{\bm{x}\in\Rmax^d}\|A\otimes \bm{x}-\bm{y}\|_{\infty}\big).
\end{equation}
That is the supremum element of the optimal set with respect to the standard partial order $\leq$ on $\Rmax^n$.
\end{theorem}

\subsection{2-norm regression}\label{2normo}

We saw in Example~\ref{inf2egg} that the max-plus $2$-norm regression problem could support multiple isolated local minima and was therefore non-convex. Because these local minima do not form a single, path connected set, this example also shows that the problem is max-plus non-convex. In the remainder of this section we develop some supporting theory then present algorithms for approximately solving the max-plus $2$-norm regression problem. 

\medskip

For $A\in\Rmax^{n\times d}$ and $\bm{x}\in\Rmax^d$ define the \emph{pattern of support} $\pattern(\bm{x})=(P_{1},\dots,P_{n}) \in \uc{P}\big(\{1,\dots,d\}\big)^n$, by
\begin{equation}
 j\in P_{i} \quad \Leftrightarrow \quad  a_{ij}+\bm{x}_{j}=(A\otimes \bm{x})_{j},
\end{equation}
for $i=1,\dots,n$. Define the \emph{domain} of a pattern $P\in \uc{P}\big(\{1,\dots,d\}\big)^n$, by 
\begin{equation}
X(P)=\{\bm{x}\in\Rmax^d~:~\pattern(\bm{x})=P\}.
\end{equation}
We say that a pattern $P$ is \emph{feasible}, if $X(P)\cap \R^d \neq \emptyset$. For a pattern $P$ define the binary relation $\tilde{\bowtie}_P$ on $\{1,\dots,d\}$, by $j\tilde{\bowtie}_P k$, if and only $j,k\in P_{i}$, for some $i=1,\dots,n$. Let $\bowtie_{P}$ denote the transitive closure of $\tilde{\bowtie}_P$ and let $|\bowtie_{P}|$ be the number of equivalence classes of $\bowtie_{P}$. If $P$ is a feasible pattern then then $X(P)$ is a set of dimension $|\bowtie_{P}|$.  From \cite[Cor.~25]{73810} we have that
\begin{equation}\label{numer}
|\{\hbox{feasible $P$} ~:~ |\bowtie_{P}|=k\}|\leq \frac{(n+d-k-1)!}{(n-k)!\cdot (d-k)! \cdot (k-1)!},
\end{equation}
for $k=1,\dots,d$, with equality for all $k$ in the generic case of a matrix $A$ with rows/cols in general position. 

Define the ordering $\preceq$ on $\uc{P}\big(\{1,\dots,d\}\big)^n$, by $P\preceq P'$, if and only if $P_{i}\subseteq P_{i}'$, for all $i=1,\dots,n$, with a strict inequality if at least one inclusion is a strict inclusion. Then the boundary of the domain $X(P)$ is given by $\cup_{\{P'~:~P\prec P' \}}X(P')$ and the closure by $\Cl\big(X(P)\big)=\cup_{\{P'~:~P\preceq P' \}}X(P')$.


Also define the \emph{feasibility matrix} by $F_{P} \in\Rmax^{d\times d}$, by
\begin{equation}\label{feasibilitymatrixdef}
f_{jk}=\left\{\begin{array}{cc} 0, & \hbox{for $j=k$,} \\ \max\big\{-\infty,\max \{a_{ik}-a_{ij}~:~ j\in P_{i}\}\big\},   & \hbox{otherwise.}\end{array}\right.
\end{equation}

\medskip

We will need to quickly review some related results to support the following Theorem. For a max-plus matrix $B\in \Rmax^{d\times d}$, the \emph{maximum cycle mean} of $B$ is defined by
\begin{equation}
\lambda(B)=\max_{\zeta}\frac{W(\zeta)}{L(\zeta)},
\end{equation}
where the maximum is taken over cycles $\zeta=\big(\zeta(1)\mapsto \dots \mapsto \zeta(k)\mapsto \zeta(1) \big)\subset \{1,\dots,d\}$. The \emph{weight} of a cycle is the sum of its edge weights $W(\zeta)=b_{\zeta(1)\zeta(2)}+\dots+b_{\zeta(k-1)\zeta(k)}+b_{\zeta(k)\zeta(1)}$ and the \emph{length} of a cycle is its total number of edges $L(\zeta)=k+1$. The \emph{Klene star} of $B\in \Rmax^d$ is defined by 
\begin{equation}
B^{\star}=\lim_{t \rightarrow \infty}\big(I\oplus B \oplus B^{\otimes 2}\oplus \dots \oplus B^{\otimes t}\big),
\end{equation}
where $I\in\Rmax^d$ is the max-plus identity matrix, with zeros on the diagonal and minus infinities off of the diagonal. From \cite[Prop.~1.6.10 and Thm.~1.6.18]{butk10} we have that if $\lambda(B)\leq0$, then $B^{\star}$ exists and $\{x\in\Rmax^d~:~ B\otimes x=x\} =\col(B^{\star})$ and that if $\lambda(B)>0$, then $B^{\star}$ does not exist and $\{x\in\Rmax^d~:~ B\otimes x=x\}\cap\R^d=\emptyset$. 

For $B\in\Rmax^{d\times d}$ define $\overline{B}\in\Rmax^{d}$ to be the arithmetic mean of the rows of $B$. It follows from \cite[Thm.~3.3]{ssb09} that $\overline{(B^{\star})}\in\relint\big(\col(B^{\star})\big).$

\begin{theorem}\label{feasibilitythm}
For $A\in\Rmax^{n\times d}$ and $P\in \uc{P}\big(\{1,\dots,d\}\big)^n$ we have
$$
\Cl\big(X(P)\big)=\{\bm{x}\in\Rmax^d~:~ F_{P}\otimes \bm{x}=\bm{x}\}.
$$
Moreover, the pattern $P$ is feasible, if and only if $\lambda(F_{P})=0$ and in the case where $P$ is feasible, we have
$$
X(P)=\relint \big( \col (F_{P}^{\star})\big)
$$
and $\overline{(F_{P}^{\star})}\in X(P)$.
\end{theorem}
\begin{proof}
First note that $\bm{x}\in \Cl\big(X(P)\big)$, if and only if $(A\otimes \bm{x})_{i}=a_{ij}+\bm{x}_{j}$, for all $j\in P_{i}$, for all $i=1,\dots,n$, which is equivalent to $\max_{k=1}^{d}(a_{ik}+\bm{x}_{k})\leq a_{ij}+\bm{x}_{j}$, for all $j\in P_{i}$, for all $i=1,\dots,n$, which is equivalent to $F_{P}\otimes \bm{x}\leq \bm{x}$ and since $F_{P}$ has zeros on its diagonal this is equivalent to $F_{P}\otimes \bm{x}=\bm{x}$. Next from \cite[Prop.~1.6.10 and Thm.~1.6.18]{butk10}, we have that $\Cl\big(X(P)\big)\cap \R^d$ is non-empty, if and only if $\lambda(F_{P})\leq 0$ and since $F_{P}$ has zeros on the diagonal this is equivalent to the condition $\lambda(F_{P})=0$. In the case that $\lambda(F_{P})=0$ we also have $ \Cl\big(X(P)\big)=\col(F_{P}^{\star})$. Then note that
$$
X(P)=\cup_{\{P'~:~P\preceq P'\}}X(P')\big/ \cup_{\{P'~:~P\prec P'\}}X(P')
$$
is equal to $\col(F_{P}^{\star})$ minus its boundary, which is precisely $\relint \big( \col (F_{P}^{\star})\big)$. The final result follows immediately from  \cite[Thm.~3.3]{ssb09}. 
\end{proof}

If there are $m$ equivalence classes in $\bowtie_{P}$ then label them arbitrarily with $\{1,\dots,m\}$ and define $c:\{1,\dots,d\}\mapsto\{1,\dots,m\}$, such that $c(j)=k$, if and only if $j$ is in the $k$th equivalence class. Now define $C\in\R ^{d\times m}$, by 
\begin{equation}
c_{jc(j)}=\frac{1}{\sqrt{|\{k~:~c(k)=c(j)\}|}}, 
\end{equation}
for $j=1,\dots,d$ and all other entries equal to zero. Then for any $\bm{x}_{P}\in \Cl\big(X(P)\big)$, we have that
\begin{equation}
\uc{A}\big(X(P)\big)=\{C\bm{h}+\bm{x}_{P}~:~\bm{h}\in\Rmax^m\}
\end{equation}
is the smallest affine subspace of $\Rmax^d$ containing $X(P)$. We could choose $\bm{x}_{P}=\overline{(F^{\star}_{P})}$ but also need to consider the case where $\bm{x}_{P}$ represents the current state of one of the algorithms that we detail later. We call $\uc{A}\big(X(P)\big)$ the \emph{extended domain} of $P$. Now define the \emph{subpattern} $\ell\in\{1,\dots,d\}^{n}$ of $P$, by $\ell(i)=\min(P_{i})$, for $i=1,\dots,n$ and define $L\in\{0,1\}^{n\times d}$, by $l_{i\ell(i)}=1$ and all other entries equal to zero. Then define the \emph{local mapping} $A_{P}:\Rmax^{d} \mapsto \Rmax^n$, by
\begin{equation}
A_{P}(\bm{x})=L\bm{x}+\bm{a}_{P},
\end{equation}
where $(\bm{a}_{P})_{i}=a_{i\ell(i)}$, for $i=1,\dots,n$. Note that $A_{P}(\bm{x})=A\otimes\bm{x}$, for all $\bm{x}\in \Cl\big(X(P)\big)$. Define the \emph{image} $Y(P)=A_{P}\big(X(P)\big)$. Then we have 
\begin{equation}
\col(A)=\bigcup_{P}Y(P),
\end{equation}
where the union is taken over all feasible patterns.  Also define the \emph{extended image} 
\begin{equation}
\uc{A}\big(Y(P)\big)=A_{P}\Big(\uc{A}\big(X(P)\big)\Big). 
\end{equation}
Note that the extended image is the smallest affine subspace containing the image and that we have 
\begin{equation}
\uc{A}\big(Y(P)\big)=\{LC\bm{h}+L\bm{x}_{P}+\bm{a}_{P}:~\bm{h}\in\Rmax^m\}.
\end{equation}
For a feasible pattern $P$ define the \emph{normal projection} map $\Phi(P,\cdot):\Rmax^n\mapsto\Rmax^n$, by
\begin{equation}
\Phi(P,\bm{y})=\arg\min\{\|\bm{y}-\bm{y}'\|_{2}~:~\bm{y}'\in\uc{A}\big(Y(P)\big)\}.
\end{equation}
Then we have 
\begin{equation}
\Phi(P,\bm{y})=LC\bm{h}^{\ast}+L\bm{x}_{P}+\bm{a}_{P},
\end{equation} 
where
\begin{equation}
\bm{h}^{\ast}=\big((LC)^{\top}LC\big)^{\dagger}(LC)^{\top}(\bm{y}-L\bm{x}_{P}-\bm{a}_{P}).
\end{equation}
Note that $LC\in\R^{n\times m}$, with 
\begin{equation}
(LC)_{i~c\circ \ell(i)}=\frac{1}{\sqrt{|\{j~:~c(j)=c\circ \ell(i)\}|}}
\end{equation}
and all other entries equal to zero. Hence we have that
\begin{equation}\label{iver}
\bm{h}^{\ast}_{k}=\sqrt{{|\{j~:~c(j)=k \}|}}~\overline{\{(\bm{y}-L\bm{x}_{P}-\bm{a}_{P})_{i} ~:~ c\circ \ell(i)=k\}},
\end{equation}
if $\{i~:~c\circ \ell(i)=k\}\neq\emptyset$ and $\bm{h}^{\ast}_{k}=0$, otherwise and where the overline in \eqref{iver} indicates taking the mean. Define the equivalence relation $\hat{\bowtie}_{P}$ on $\{1,\dots,n\}$, by $i\hat{\bowtie}_{P}i'$, if and only if $\ell(i)\bowtie_{P} \ell(i')$, then we have
\begin{equation}
\Phi(P,\bm{y})_{i}=\overline{\{(\bm{y}-L\bm{x}_{P}-\bm{a}_{P})_{i'}~:~ i\hat{\bowtie}_{P} i'\}}+(L\bm{x}_{P}+\bm{a}_{P})_{i},
\end{equation}
for $i=1,\dots,n$. Also define 
\begin{align}
A^{-1}_{P}\big(\Phi(P,\bm{y})\big) &=\{\bm{x}\in\uc{A}\big(X(P)\big)~:~A_{P}(\bm{x})=\phi(P,\bm{y})\} \\ &=C\bm{h}^{\ast}+C\ker(LC)+\bm{x}_{P},
\end{align}
where 
\begin{equation}
(C\bm{h}^{\ast})_{j}=\overline{\{(\bm{y}-L\bm{x}_{P}-\bm{a}_{P})_{i}~:~ c\circ p(i)=c(j)\}},
\end{equation}
for $j=1,\dots,d$ and
\begin{equation}
C~\hbox{ker}(LC)=\hbox{span}\{\underline{e}_{j}\in\Rmax^d~:~ j\in \{1,\dots,d\}/\support(P) \},
\end{equation}
where $\support(P)\subset\{1,\dots,d\}$ is the \emph{support} of $P$, defined by $\support(P)=\cup_{i=1}^{n}P_{i}$. We say that $\Phi(P,y)$ is \emph{admissible}, if $\Phi(P,\bm{y})\in \Cl\big(Y(P)\big)$, or equivalently, if $A^{-1}_{P}\big(\Phi(P,\bm{y})\big)\cap \Cl\big(X(P)\big)\neq \emptyset$. Also define the \emph{closest local minimum} map $\Psi(P,\bm{y},\cdot):\Rmax^d\mapsto\Rmax^d$, by 
\begin{equation}
\Psi(P,\bm{y},\bm{x})=\arg\min\{\|\bm{x}-\bm{x}'\|_{2}~:~\bm{x}'\in A^{-1}_{P}\big(\Phi(P,\bm{y})\big)\},
\end{equation}
which is given by $\Psi(P,\bm{y},\bm{x})_{j}=(C\bm{h}^{\ast})_{j}$, for $j\in\support(P)$ and $\Psi(P,\bm{y},\bm{x})_{j}={x}_{j}$, otherwise. 

\begin{theorem}\label{admissthm}For $A\in\Rmax^{n\times d}$, $\bm{y}\in\Rmax^n$ and a feasible pattern $P\in \uc{P}\big(\{1,\dots,d\}\big)^n$ the normal projection $\Phi(P,\bm{y})$ is admissible, if and only if 
\begin{equation}\label{thmcondition}
F_{P}\otimes \Psi(P,\bm{y},\underline{-\infty})= \Psi(P,\bm{y},\underline{-\infty}),
\end{equation}
where $\underline{-\infty}\in\Rmax^d$ is a vector with all entries equal to $-\infty$.
\end{theorem}
\begin{proof}
If \eqref{thmcondition} holds then from Theorem~\ref{feasibilitythm} we have $\Psi(P,\bm{y},\underline{-\infty})\in\Cl\big(X(P)\big)$ and therefore $\Phi(P,\bm{y})$ is admissible. Conversely suppose that $\Psi(P,\bm{y})$ is admissible, then there exists $\bm{x}\in A^{-1}_{P}\big(\Phi(P,\bm{y})\big)$ such that $F_{P}\otimes \bm{x}=\bm{x}$. Note that for $j\in\{1,\dots,d\}/\support(P)$ we have $(F_{P}\otimes \bm{x}')_{j}=x_{j}'$, for all $\bm{x}'\in\Rmax^d$. Also note that for $k\in\support(P)$ we have that $(F_{P}\otimes \bm{x}')_{k}$ is non-decreasing in $x_{j}'$, for all $j=1,\dots,d$. Therefore 
$$
\big(F_{P}\otimes \Psi(P,\bm{y},\underline{-\infty})\big)_{k} \leq \big(F_{P}\otimes \bm{x}\big)_{k}\leq x_{k}=\Psi(P,\bm{y},\underline{-\infty})_{k},
$$
for $k\in\support(P)$ and
$$
\big(F_{P}\otimes \Psi(P,\bm{y},\underline{-\infty})\big)_{j}=\Psi(P,\bm{y},\underline{-\infty})_{j},
$$
for $j\in\{1,\dots,d\}/\support(P)$. Therefore \eqref{thmcondition} holds. 

\end{proof}

For $A\in\Rmax^{n\times d}$ and $\bm{y}\in\Rmax^n$ define the \emph{squared residual} $R:\Rmax^d\mapsto\mathbb{R}_{+}$, by $R(\bm{x})=\|A\otimes\bm{x} -\bm{y}\|_{2}^{2}/2$. For a feasible pattern $P$ define the \emph{local squared residual} $R_{P}:\uc{A}\big(X(P)\big)\mapsto \mathbb{R}$, by 
\begin{equation}
R_{P}(\bm{x})=\|A_{P}(\bm{x})-\bm{y}\|_{2}^2/2=\|L\bm{x}+\bm{a}_{P}-\bm{y}\|_{2}^2/2.
\end{equation}
Note that for $\bm{x}\in\Cl\big(X(P)\big)$, we have $R(\bm{x})=R_{P}(\bm{x})$. Hence $R$ is piecewise quadratic.

\begin{example}\label{2negg}

Consider 
$$
A=\left[\begin{array}{cc} 0 & 0 \\ 1 & 0 \\ 0 & 1 \end{array}\right], \quad \bm{y}=\left[\begin{array}{cc} 0  \\ 0.5  \\ 0 \end{array}\right], \quad \bm{y}'=\left[\begin{array}{cc} 0  \\ 1.5  \\ 2 \end{array}\right].
$$
There are seven feasible patterns, their domains are displayed in Figure~\ref{Pdomains}.

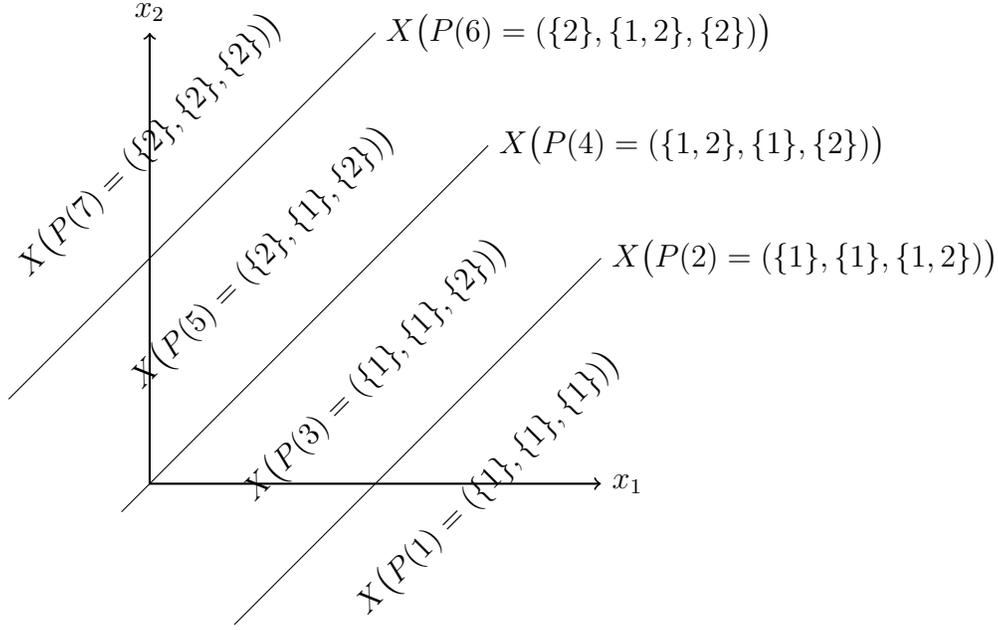
\begin{figure}
\begin{center}
\begin{tikzpicture}[scale=3]
    \draw [<->,thick] (0,2) node (yaxis) [above] {$x_2$}
        |- (2,0) node (xaxis) [right] {$x_1$};
    
    \draw (1.5,0) node [rotate=45] {$X\big(P(1)=(\{1\},\{1\},\{1\})\big)$};
    
    \draw (0.5-0.125,-0.5-0.125)--(2,1) node [right] {$X\big(P(2)=(\{1\},\{1\},\{1,2\})\big)$};
    
    \draw (1,0.5) node [rotate=45] {$X\big(P(3)=(\{1\},\{1\},\{2\})\big)$};
    
    \draw (-0.125,-0.125)--(1.5,1.5) node [right]{$X\big(P(4)=(\{1,2\},\{1\},\{2\})\big)$};
    
    \draw (0.5,1) node [rotate=45] {$X\big(P(5)=(\{2\},\{1\},\{2\})\big)$};
    
    \draw (-0.5-0.125,0.5-0.125)--(1,2) node [right] {$X\big(P(6)=(\{2\},\{1,2\},\{2\})\big)$};
    
     \draw (0,1.5) node [rotate=45] {$X\big(P(7)=(\{2\},\{2\},\{2\})\big)$};
    
    
  \end{tikzpicture}
  \end{center}
\caption{Domains of feasible patterns for the matrix $A$ of Example~\ref{2negg}.}\label{Pdomains}
\end{figure}

For $P=P(3)=\big(\{1\},\{1\},\{2\}\big)$, we have
$$
F_{P}=\left[\begin{array}{cc} 0 & 0 \\ -1 & 0\end{array}\right].
$$
Since $\lambda(F_{P})=0$, we have that $P$ is admissible and its domain is given by $X\big(P\big)=\relint\big(\col(F_{P}^{\star})\big)$. In this case $F_{P}^{\star}=F_{P}$ and 
$$
\col(F_{P}^{\star})=\{\bm{x}\in\Rmax^2~:~ x_{2}+1\geq x_{1}\geq x_{2} \}, \quad \relint\big(\col(F_{P}^{\star})\big)=\{\bm{x}\in\Rmax^2~:~ x_{2}+1> x_{1}> x_{2}\}.
$$
The boundary of $X(P)$ is given by $X\big(P(2)\big)\cup X\big(P(4)\big)$. These patterns are both feasible and
\begin{align*}
F_{P(2)}=\left[\begin{array}{cc} 0 & 1 \\ -1 & 0\end{array}\right], \quad &X\big(P(2)\big) = \relint\big(\col(F_{P(2)}^{\star})\big)=\col(F_{P(2)}^{\ast})=\{\bm{x}\in\Rmax^2~:~ x_{1}= x_{2}+1\}, \\ F_{P(4)}=\left[\begin{array}{cc} 0 & 0 \\ 0 & 0\end{array}\right], \quad &X\big(P(4)\big) = \relint\big(\col(F_{P(4)}^{\star})\big)=\col(F_{P(4)}^{\ast})=\{\bm{x}\in\Rmax^2~:~ x_{1}= x_{2}\}.
\end{align*}

\medskip

For $P=\big(\{1,2\},\{2\},\{1\}\big)$, we have
$$
F_{P}=\left[\begin{array}{cc} 0 & 1 \\ 1 & 0\end{array}\right].
$$
Since $\lambda(F_{P})=1$, we have that $P$ is not admissible. 

\medskip

Note that $\bm{z}\in\col(A)$, if and only if $\bm{z}\otimes \alpha=[z_{1}+\alpha,z_{2}+\alpha,z_{3}+\alpha]^{\top}\in\col(A)$, for all $\alpha\in\Rmax$. Therefore $\col(A)$ has translational symmetry in the $[1,1,1]^{\top}$ direction. Similarly for all of the pattern images and extended images. We can therefore study these objects by examining their image under the orthogonal projection $\Pi:\Rmax^3\mapsto\{\bm{z}\in\Rmax^3 ~:~[1,1,1]\bm{z}=0\}$. This is the same idea as in the tropical projected space $\mathbb{TP}^2$, which is usually taken to be a projection onto $\{\bm{y}\in\Rmax^3 ~:~[1,0,0]\bm{y}=0\}$. However the choice of projection we use here is  more convenient for analyzing the $2$-norm regression problem. The projected pattern images,  extended images and a sample of normal projections are displayed in Figure~\ref{Pimages}.

\medskip

Returning our attention to the pattern $P=P(3)=\big(\{1\},\{1\},\{2\}\big)$, the equivalence relation $\bowtie_{P}$ is the identity relation, so that $C$ is the $2\times 2$ identity matrix and the extended domain is given by
$$
\uc{A}\big(X(P)\big)=\{C\bm{h}+\bm{x}_{P}~:~ \bm{h}\in \Rmax^2\}=\Rmax^2.
$$
We have $p=(1,1,2)$, so the local map is given by $A_{P}(\bm{x})=[x_{1},x_{1}+1,x_{2}+1]^{\top}$. The extended image is given by
$$
\uc{A}\big(Y(P)\big)=\{\bm{z}\in\Rmax^3~:~ z_{2}=z_{1}+1 \},
$$
and the image is given by
$$
Y(P)=\{\bm{z}\in\Rmax^3~:~ z_{3}>z_{1}>z_{3}-1 \}.
$$
The boundary of the image is given by $Y\big(P(2)\big)\cup Y\big(P(4)\big)$, where
$$
Y\big(P(2)\big)=\{\bm{z}\in\Rmax^3~:~ z_{1}=z_{3}-1 \}, \quad Y\big(P(2)\big)=\{\bm{z}\in\Rmax^3~:~ z_{1}=z_{3} \}.
$$

\medskip

The equivalence relation $\hat{\bowtie}_{P}$ has equivalence classes $\{1,2\}$ and $\{3\}$. Using $\bm{x}_{P}=[0,0]^{\top}$, we obtain $\Phi(P,\bm{y})=[-0.25,0.75,0]^{\top}$. We have $\support(P)=\{1,2\}$, so that $\ker(LC)=\{0\}$ and therefore 
$$
A_{P}^{-1}\big(\Phi(P,\bm{y})\big)=C\bm{h}^{\ast}=[-0.25,-1]^{\top}.
$$
Similarly we must have $\Psi(P,\bm{y},\underline{-\infty})=[-0.25,-1]^{\top}$. Checking
$$
F_{P}\otimes \Psi(P,\bm{y},\underline{-\infty})=\left[\begin{array}{cc} 0 & 0 \\ -1 & 0\end{array}\right]\otimes\left[\begin{array}{cc} -0.25 \\ -1 \end{array}\right]=\left[\begin{array}{cc} -0.25 \\ -1 \end{array}\right]= \Psi(P,\bm{y},\underline{-\infty}),
$$
we show that $\Phi(P,\bm{y})$ is admissible. 

For the alternative target vector $\bm{y}'$ we have  $\Phi(P,\bm{y}')=[0.25,1.25,2]^{\top}$,
$$
A_{P}^{-1}\big(\Phi(P,\bm{y})\big)=C\bm{h}^{\ast}=[0.25,1]^{\top}
$$
and 
$$
F_{P}\otimes \Psi(P,\bm{y},\underline{-\infty})=\left[\begin{array}{cc} 0 & 0 \\ -1 & 0\end{array}\right]\otimes\left[\begin{array}{cc} 0.25 \\ 1 \end{array}\right]=\left[\begin{array}{cc} 1 \\ 1 \end{array}\right]> \Psi(P,\bm{y},\underline{-\infty}),
$$
shows that $\Phi(P,\bm{y}')$ is not admissible. 

\begin{figure}
\begin{center}
\begin{tikzpicture}[scale=4]
    \draw [<->,thick] (0,1) node (yaxis) [above] {$(-x_2+x_{3})/\sqrt{2}$}
        |- (1,0) node (xaxis) [below] {$(-2 x_1+x_{2}+x_{3})/\sqrt{6}$};

    \draw[thick,fill] (-0.7071,0.4082) circle [radius=0.02cm] node [below right] {$Y_{1/2}$};
   
    \draw (-0.7071,0.4082)--(0,0.8165)[line width=2];

	\draw ( -0.3535, 0.6123)--( -0.3535, 0.6123) node [below right] {$Y_3$};
   
    \draw[thick,fill] (0,0.8165) circle [radius=0.02cm] node [ right=0.4cm ] {$Y_{4}$};   
    
    \draw ( 0.3535, 0.6123)--( 0.3535, 0.6123) node [below left] {$Y_5$};
    
  \draw[thick,fill] (0.7071,0.4082) circle [radius=0.02cm] node [below left] {$Y_{6/7}$};
    
    \draw (-0.7071,0.4082)--(0,0.8165)[line width=2];

    \draw (0,0.8165)--(0.7071,0.4082) [line width=2];
   
    \draw ( -1.4142 ,  -0.0001)--(0.7954    , 1.2759) [line width=0.5]; 
   
   \draw ( -1.0606   , 0.2041)--( -1.0606   , 0.2041) node[above left ] {$\uc{A}(Y_3)$};

     \draw ( 1.4142 ,  -0.0001)--(-0.3535   , 1.0207) [line width=0.5];

   \draw ( 1.0606   , 0.2041)--( 1.0606   , 0.2041) node[above right ] {$\uc{A}(Y_5)$};

  \draw[thick,fill] ( -0.3536, 0.2041) circle [radius=0.02cm] node [below] {$\bm{y}$};
  
    \draw( -0.3536, 0.2041)-- ( -0.5303, 0.5103) [->,thick,dashed]node [above left] {$\Phi\big(P(3),\bm{y}\big)$};

  \draw[thick,fill] ( +0.3536, 1.4289)circle [radius=0.02cm] node [above] {$\bm{y}'$};
  
      \draw( +0.3536, 1.4289) -- ( 0.5303, 1.1227) [->,thick,dashed]node [below right] {$\Phi\big(P(3),\bm{y}'\big)$};


  \end{tikzpicture}
  \end{center}
\caption{For the problem of example  Example~\ref{2negg}. Projected pattern images and extended images for the matrix $A$. $Y_{1/2}=Y\big(P(1)\big)=\uc{A}\Big(Y\big(P(1)\big)\Big)=Y\big(P(2)\big)=\uc{A}\Big(Y\big(P(2)\big)\Big)$, $Y_3=Y\big(P(3)\big)$, $Y_4=Y\big(P(4)\big)=\uc{A}\Big(Y\big(P(4)\big)\Big)$, 
$Y_5=Y\big(P(5)\big)$ and $Y_{6/7}=Y\big(P(6)\big)=\uc{A}\Big(Y\big(P(6)\big)\Big)=Y\big(P(7)\big)=\uc{A}\Big(Y\big(P(7)\big)\Big)$. Target vectors $\bm{y}$ and $\bm{y}'$ with normal projections onto $\uc{A}(Y_{3})$.
}\label{Pimages}
\end{figure}
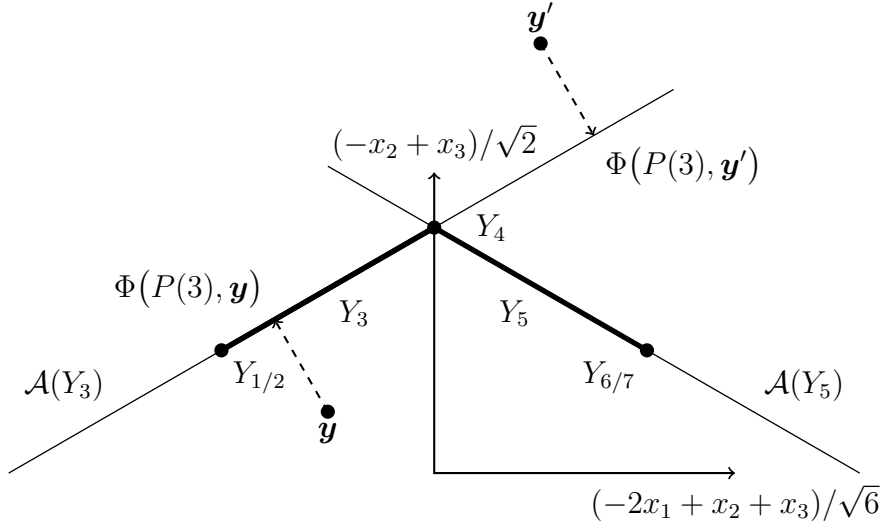

\end{example}

\subsubsection{Steepest descent method}

For $A\in\Rmax^{n\times d}$, $\bm{y}\in\Rmax^n$, $\bm{x}\in\Rmax^d$ and any feasible pattern $P$, with $P\preceq \pattern(\bm{x})$, define the \emph{subgradient} $\nabla(\bm{x},P)$, by
\begin{equation}
\nabla(\bm{x},P)=\frac{d R_{P}(\bm{x})}{d\bm{x}}\Big|_{\uc{A}\big(X(P)\big)}.
\end{equation}
Since $C$ is orthogonal we have that
\begin{equation}
\nabla(\bm{x},P)_{j}=CC^{\top}\frac{d R_{P}(\bm{x})}{d\bm{x}}.
\end{equation}
So that
\begin{equation}
\nabla(\bm{x},P)_{j}= \frac{\sum_{\{i~:~c\circ \ell(i)=c(j)\}}a_{i\ell(i)}+x_{\ell(i)}-y_{i}}{|\{k~:~c(k)=c(j)\}|},
\end{equation}
if $\{i~:~c\circ \ell(i)=c(j)\}\neq\emptyset$ and $\nabla(\bm{x},P)_{j}=0$, otherwise. We say that a subgradient $\nabla(\bm{x},P)$ is \emph{admissible}, if there exists $\epsilon>0$, such that $\bm{x}-\mu\nabla(\bm{x},P)\in X(P)$, for all $\mu\in(0,\epsilon]$. It is easy to show that the subgradient  $\nabla(\bm{x},P)$ is admissible, if and only if
\begin{equation}
\nabla(\bm{x},P)_{j}>\nabla(\bm{x},P)_{k},
\end{equation}
whenever $j,k\in\pattern(\bm{x})_{i}$, $j\in P_{i}$ and $k\not\in P_{i}$, for some $i=1,\dots,n$, for all $j,k=1,\dots,d$. It therefore follows that $\nabla\big(\bm{x},\pattern(\bm{x})\big)$ is always admissible.  

Define the \emph{steepest descent gradient field} $F:\Rmax^d\mapsto \Rmax^d$, by $F(\bm{x})=-\nabla(\bm{x},P^{\ast})$, where 
\begin{equation}
P^{\ast}=\arg\max\{\|\nabla(\bm{x},P)\|_{2}~:~P\preceq \pattern(\bm{x}), ~ \hbox{$\nabla(\bm{x},P)$ is admissible}\}.
\end{equation}
We say that $\gamma:[0,T)\mapsto \Rmax^d$ is a \emph{steepest descent path} if 
\begin{equation}\label{spath}
\lim_{h\rightarrow 0_{+}}\frac{\gamma(t+h)-\gamma(t)}{h}=F\big(\gamma(t)\big),
\end{equation}
for all $t\in[0,T)$. For a pattern $P$ define the operator $\phi_{P}~:~\Rmax^d\times \R_{+}\mapsto \Rmax^d$, by 
\begin{equation}
\phi_{P}(\bm{x},0)=\bm{x}, \quad \frac{d\phi_{P}(\bm{x},t)}{dt}=-\nabla(\phi_{P}(\bm{x},t),P).
\end{equation}
Then we have
\begin{equation}
\lim_{t\rightarrow\infty}\phi_{P}(\bm{x},t)=\Psi(P,\bm{y},\bm{x})
\end{equation}
and
\begin{equation}
\phi_{P}(\bm{x},t)_{j}=x_{j}+\Big(1-\exp\big(-t|\{i~:~c\circ \ell(i)=c(j)\}|\big)\Big)\big(\Psi(P,\bm{y},\bm{x})-\bm{x}\big)_{j}.
\end{equation}
Also define 
\begin{align}
t^{\ast}& =\inf\{t>0~:~ \phi(\bm{x},t)\not\in X(P)\} \\ &=\min_{i=1,\dots,n, ~ j\not\in P_{i}} \inf\{t>0~:~ a_{ij}+\phi(\bm{x},t)_{j}=a_{i\ell(i)}+\phi(\bm{x},t)_{\ell(i)}\}.
\end{align}
Algorithm~\ref{steep} constructs a steepest descent path as a sequence of smooth segments. For each smooth segment of the path, the algorithm must determine whether the smooth flow reaches the fixed point $\Psi(P,\bm{y},\bm{x})$, in which case $t^{\ast}=\infty$, or leaves the domain of the current pattern, in which case $t^{\ast}$ is finite.

\begin{algorithm}[h]
\caption{ \label{steep} {\bf (Steepest descent)} 
Given an initial guess $\bm{x}$, returns an locally optimal solution to Problem~\ref{mppnreg} with $p=2$.}

\medskip

\begin{algorithmic}[1]

\While{not converged}
\State compute $P^{\ast}$ and $F(\bm{x})$
\If {$F(\bm{x})=0$} converged
\Else \State compute $t^{\ast}$
\If {$t^{\ast}=\inf$} update $\bm{x}\mapsfrom \Psi(P^{\ast},\bm{y},\bm{x})$
\Else 
\State update $\bm{x}\mapsfrom \phi_{P^{\ast}}(\bm{x},t^{\ast})$
\EndIf
\EndIf
\EndWhile
\State return solution $\bm{x}$

\end{algorithmic}

\end{algorithm}

The worst case cost of Algorithm~\ref{steep} comes from computing $P^{\ast}$ and $F(\bm{x})$. To do this we need to examine all patterns $P$, with $P\preceq \pattern(\bm{x})$, and there can be exponentially many of these. For example take $\bm{x}\in\Rmax^d$, with $\bm{x}_{j}=-a_{1j}$, then $\pattern(\bm{x})_{1}=\{1,\dots,d\}$ and there are at least $2^d-1$ patterns $P$ with $P\preceq \pattern(\bm{x})$. Computing the subgradient and determining the admissibility of an individual pattern has cost $\uc{O}(nd)$. Computing $t^{\ast}$ has cost $\uc{O}(nd)$. Therefore if Algorithm~\ref{steep} generates a sequence of $k$ smooth path segments and needs to check an average of $m$ subgradients on each calculation of $F(\bm{x})$, then the total cost is $\uc{O}(kmnd)$.

It therefore appears that finding the steepest descent direction is potentially very computationally expensive. In fact we can show that the easier problem of determining whether any direction that reduces the residual exists is NP-hard. To do this we show that the following problem, known as the set covering problem, can be solved by determining whether or not the zero vector is a local minimum of the residual for a $\big(n+m+m(m-1)/2+1\big)\times m$ max-plus $2$-norm regression problem. See Theorem~\ref{setsetset}.

\begin{problem}\label{setcovering}
Let $F=\big\{F_{i}\subset\{1,\dots,n\}~:~i=1,\dots,m\big\}$ be a family of subsets with $\cup_{i=1}^{m}F_{i}=\{1,\dots,n\}$ and let $1<k<m$. Does there exist a subset $\{j(1),\dots,j(k)\}\subset\{1,\dots,m\}$, such that  $\cup_{i=1}^{k}F_{j(i)}=\{1,\dots,n\}$?
\end{problem}

\subsubsection{Newton's method}

The results of the previous section and Theorem~\ref{setsetset} suggest that computing exact local minima for the max-plus $p=2$ regression problem might not be computationally feasible for larger problems. Instead we propose using the following technique, which consists of Newton's method with an undershooting parameter.  

Recall that the squared residual $R:\Rmax^d\mapsto\mathbb{R}_{+}$, given by $R(\bm{x})=\|A\otimes\bm{x} -\bm{y}\|_{2}^{2}/2$ is picewise quadratic and that  for $\bm{x}\in\Cl\big(X(P)\big)$, we have $R(\bm{x})=R_{P}(\bm{x})$.
Newton's method minimizes a function by iteratively mapping to the minimum of a local quadratic approximation to that function. In the case of the squared residual $R$ this means iteratively mapping to the minimum of the locally quadratic piece. There are several options when implementing Newton's method, for example when $\bm{x}$ is contained in the closure of more than one domain, which pattern do we choose? Also, how do we choose between non-unique minima? The method we set out below is chosen primarily for its simplicity. 
\medskip

For $\bm{x}\in\Rmax^d$, define the \emph{subpattern} $p(\bm{x})\in\{1,\dots,d\}^n$, by 
\begin{equation}
p(\bm{x})_{i}=\min\big(\pattern(\bm{x})_{i}\big)=\min\{j~:~(A\otimes \bm{x})_{i}=a_{ij}+x_{j}\}, \quad i=1,\dots,n.
\end{equation}
Then $p(\bm{x})\preceq \pattern(\bm{x})$ and $\bm{x}\in\Cl\Big(X\big(p(\bm{x})\big)\Big)$.
Define the \emph{Newton update map} $\mathcal{N}:\Rmax^d\mapsto \Rmax^d$, by
\begin{equation}\label{newtonmappp}
\mathcal{N}(\bm{x})=\Psi\big(p(\bm{x}),\bm{y},\bm{x}\big).
\end{equation}
The map \eqref{newtonmappp} is set to always chooses a pattern whose domain is of the maximum possible dimension. In the case where the minima is non-unique, it returns the one that is closest to the current point.

A difficulty for Newton's method is that the non-differentiability of $R$ means that the iteration needn't converge to a local minima and can instead get caught in a periodic orbit. This makes choosing a stopping condition difficult. We use the rule that if the residual has not decreased in some fixed number of steps then we terminate the algorithm and return the best solution from the iterations orbit. We also include a shooting parameter $\mu\in(0,1)$ and make the update $\bm{x}\mapsfrom (1-\mu)\bm{x}+\mu \mathcal{N}(\bm{x})$. Choosing $\mu<1$ causes the method to undershoot and so avoid being caught in the periodic orbits mentioned previously. If Algorithm~\ref{newton} iterates $k$ times then it has cost $\uc{O}(knd)$. In the Numerical examples that follow we randomly sample ten different initial conditions then apply Algorithm~\ref{newton} once with $\mu=1$ then once more with $\mu=0.05$, each time using $t=5$, then pick the best approximate solution. Optimizing the choice of parameters and random starting conditions is an important topic for future research. 

\begin{algorithm}
\caption{ \label{newton} {\bf (Newton's method)} 
Given an initial guess $\bm{x}$, returns an approximate solution to Problem~\ref{mppnreg} with $p=2$. Parameters are $t\in\mathbb{N}$ the number of iterations for stopping condition and $\mu\in(0,1)$ the undershooting parameter, which may be allowed to vary during the computation.}

\medskip

\begin{algorithmic}[1]

\State set $r_{\min}=\infty$
\While{not terminated}
\State update $\bm{x}\mapsfrom (1-\mu)\bm{x}+\mu \mathcal{N}(\bm{x})$
\If {$R(\bm{x})<r_{\min}$} $r_{\min}=R(\bm{x})$, $\hat{\bm{x}}=\bm{x}$, \EndIf
\If {$r_{\min}$ not decreased for $t$ iterations} terminate
\EndIf
\EndWhile
\State return approximate solution $\hat{\bm{x}}$

\end{algorithmic}

\end{algorithm}

\section{Time-series analysis}\label{tss}

Consider the $d$-dimensional stochastic max-plus linear dynamical system 
\begin{equation}\label{dysys}
\bm{x}(n+1)=M\otimes \bm{x}(n)+\zeta(n),
\end{equation}
where $M\in\Rmax^{d\times d}$ and $\zeta(0),\zeta(1),\dots\in\R^d$ are i.i.d Gaussians with mean zero and covariance matrix $\sigma^2 I$. Suppose that we do not know $M$, but that we have observed an orbit $\bm{x}(0),\bm{x}(1),\dots,\bm{x}(N)$ and want to estimate $M$ from this data. The maximum likelihood estimate for this inference problem is given by
\begin{equation}\label{invdsres}
\min_{A\in\Rmax^{d\times d}}\mathbb{P}\{\bm{x}(0),\bm{x}(1)\dots,\bm{x}(N)~|~ \bm{x}(k+1)=A\otimes \bm{x}(k)+\zeta(k), ~ k=0,1,\dots,N-1 \}.
\end{equation}
This problem can be expressed as $d$ independent regression problems as follows. Expanding \eqref{invdsres} yields
\begin{align}
\mathbb{P}\{\bm{x}(0),\dots,\bm{x}(N) ~|~ A\} &=\prod_{n=0}^{N-1}\mathbb{P}\{\zeta(n)=\bm{x}(n+1)-A\otimes \bm{x}(n)\} \\  &=\prod_{n=0}^{N-1}\prod_{k=1}^{d}\frac{1}{\sqrt{2\pi\sigma^2}}\exp \left({\frac{-\Big(\bm{x}(n+1)-A\otimes \bm{x}(n)\Big)_{k}^2}{2\sigma^2}}\right).
\end{align}
The log likelihood is therefore given by
\begin{equation}\label{llhood}
\log\big(\mathbb{P}\{\bm{x}(0),\dots,\bm{x}(N) ~|~ A\}\big) =\frac{-Nd}{2}\log(2\pi\sigma^2)+\frac{1}{2\sigma^2}\|A\otimes X(:,1:N)-X(:,2:N+1)\|_{F}^{2},
\end{equation}
where $X\in\Rmax^{d\times (N+1)}$ is the matrix whose columns are the time series observations $\bm{x}(0),\dots,\bm{x}(N)$ and where we use the Matlab style notation $X(\mathcal{I},\mathcal{J})$ to indicate the submatrix of formed from the intersection of the $\mathcal{I}$ rows and $\mathcal{J}$ columns of $X$ and use the symbol $:$ alone to denote the full range of row/cols. Next note that
\begin{equation}\label{ffffff}
\|A\otimes X(:,1:N)-X(:,2:N+1)\|_{F}^{2}=\sum_{k=1}^{d}\|A(k,:)\otimes X(:,1:N)-X(k,2:N+1)\|_2^2.
\end{equation}
Minimizing \eqref{invdsres} is therefore equivalent to minimizing each of the terms summed over in \eqref{ffffff}. The $k$th of these terms measures our model's ability to predict the value of the $k$th variable at the next time step. To minimize this error we choose the $k$th row of $A$ by
\begin{equation}\label{rowA}
A(k,:)=\arg\min_{\bm{x}\in\Rmax^{1\times d}}\|\bm{x}\otimes X(:,1:N) -X(k,2:N+1)\|_{2},
\end{equation}
which requires us to solve an $n\times d$ max-plus $2$-norm regression problem. We can therefore solve \eqref{invdsres} by solving $d$ such regression problems. 

\begin{example}\label{invpegg}
Consider the matrix
$$
M=\left[\begin{array}{cccc} 7 &  15 &  10 &  -\infty   \\  14 &  -\infty &  11 &  11  \\  14 &  -\infty &  -\infty &  -\infty   \\  15 &  8 &  7 &  9  \\  \end{array}\right].
$$
From the initial condition $\bm{x}(0)=[0,0,0,0]^{\top}$ we generate two orbits of length $n=200$ by iterating \eqref{dysys}. One with a low noise level, $\sigma=1$ and one with a  high noise level, $\sigma=5$. Next we compute the maximum likelihood estimate for $M$ from the time series, by applying Algorithm~\ref{newton} to each of the row problems \eqref{rowA}, for $k=1,\dots,d$. Our estimates are given by
$$ 
  A(\sigma=1)=\left[\begin{array}{cccc} 2.65 &  14.9 &  10.6 &  10  \\  13.8 &  -30.4 &  -53.2 &  10.9  \\  14 &  7.06 &  8.5 &  5.7  \\  15 &  9.4 &  -50.3 &  8.28  \\  \end{array}\right], \quad A(\sigma=5)=\left[\begin{array}{cccc} 8.24 &  14.1 &  11 &  1.67  \\  13.8 &  -\infty &  9.28 &  11.1  \\  13.8 &  -\infty &  -\infty &  -\infty   \\  14.3 &  9.21 &  7.49 &  7.62  \\  \end{array}\right].
 $$
 Table~\ref{resftab} displays the Frobenius error term \eqref{ffffff} for each of these estimates. Note that both of these estimates fit the data better than the true system matrix $M$, which indicates that Algorithm~\ref{newton} is able to find close to optimal solutions to the regression problem. 
 
Comparing our estimates to $M$, we see that in both cases we have inferred values that are roughly correct for the larger entries in the matrix but that the minus infinities are poorly approximated in both cases and that some of the smaller finite entires are poorly approximated in the low noise case. For each orbit we record the matrix $S\in\mathbb{N}^{4\times 4}$, with
$$
s_{ij}=|\{0\leq n<N-1 ~:~\big(A\otimes x(n)\big)_{i}=a_{ij}+x(n)_{j}\}|,
$$
which records how often variable $j$ attains the maximum in determining variable $i$ at the next time step, for $i,j=1,\dots,4$. Therefore $s_{ij}$ can be  thought of as a measure of how much evidence we have to infer the parameter $a_{ij}$ from the orbit. These matrices are given by
$$
S(\sigma=1)=\left[\begin{array}{cccc} 0 &  201 &  0 &  0  \\  167 &  0 &  5 &  29  \\  201 &  0 &  0 &  0  \\  197 &  0 &  0 &  4  \\  \end{array}\right], \quad S(\sigma=5)=\left[\begin{array}{cccc} 30 &  137 &  34 &  0  \\  108 &  0 &  41 &  52  \\  201 &  0 &  0 &  0  \\  133 &  30 &  12 &  26  \\  \end{array}\right].
$$
Comparing the results it is clear that our inferences are more accurate for entries with more evidence. In the low noise case the evidence is all contained on a small number of entries, as under the nearly deterministic behavior of this regime only a few positions are ever able to attain the maximum. In the high noise case the more random behavior means that more entries are able to attain the maximum and therefore the evidence is more uniformly distributed, except onto the minus infinity entries, which can never attain the maximum.

\end{example}

Inferring the values of entries that do not play a role in the dynamics or only play a very small role is therefore an ill posed problem and consequently we obtain MLE matrices $A$ that do a good job of fitting the data but which are not close to the true system matrix $M$. There are two common strategies for coping with such ill posed inverse problems. The first is to choose a prior distribution for the inferred parameters, then compute a maximum a posteri estimate which minimizes the likelihood times the prior probability. The second approach is to add a regularization penalty to the targeted residual. Regularization is typically used to improved the well-posedness of inverse problems and to promote solutions which are in some way simpler. Typical choices for conventional linear regression problems are the $1$-norm or $2$-norm of the solution. For max-plus linear $2$-norm regression we propose the following regularization penalty, which is chosen to promote solutions $\bm{x}\in\Rmax^d$ with smaller entries and with more entries equal to $-\infty$. 

\begin{problem}\label{regl2}
For $A\in\mathbb{R}_{\max}^{n\times d}$, $\bm{y}\in\mathbb{R}_{\max}^{n}$ and $\lambda\geq 0$, we seek
\begin{equation}\label{qqq1}
\min_{\bm{x}\in\mathbb{R}_{\max}^{d}}\Big(\|\big(A\otimes \bm{x}\big)-y\bm{y}\|_{2}^2+\lambda\sum_{j=1}^{d}\bm{x}_{j}\Big).
\end{equation}
\end{problem}

Problem~\ref{regl2} can be solved by solving a sequence of max-plus $2$-norm regression problems by an approach which is inspired by the iteratively reweighed least squares method for solving conventional $1$-norm regularized $2$-norm regression problems.  Let $I\in\Rmax^{d\times d}$ be the max-plus identity matrix with zeros on the diagonal and minus infinities off of the diagonal and consider the residual
\begin{align}
\label{ww1} \left\|\left[\begin{array}{c} A \\ I \end{array}\right]\otimes \bm{x}-\left[\begin{array}{c} \bm{y} \\ \bm{x}'-\lambda/2 \end{array}\right]\right\|_{2}^2&=\|A\otimes \bm{x}-\bm{y}\|_{2}^{2}+\sum_{j=1}^{d}(\lambda/2+\bm{x}-\bm{x}')^2 \\ &=\|A\otimes \bm{x}-\bm{y}\|_{2}^{2}+\lambda\sum_{j=1}^{d}(\bm{x}-\bm{x}')_{j}+\uc{O}\big((\bm{x}-\bm{x}')^2\big).
\end{align}
Therefore, for $\bm{x}$ close to $\bm{x}'$, \eqref{ww1} only differs from \eqref{qqq1} by a constant factor. Algorithm~\ref{irsls} computes a sequence of approximate solutions, each time shifting the  additional target variables in \eqref{ww1} to match the gradient of the residual in \eqref{qqq1}. If a component of the solution appears to be diverging to minus infinity, then we set it to equal this limit.

\begin{algorithm}
\caption{ \label{irsls} {\bf (Iteratively reshifted least squares)} 
Given an initial guess $\bm{x}$, returns an approximate solution to Problem~\ref{regl2}.}

\medskip

\begin{algorithmic}[1]

\While{not converged}
\State apply a max-plus 2-norm regression solver to compute
 $$\bm{x}\mapsfrom \arg\min_{\bm{x}'\in\Rmax^d}\left\|\left[\begin{array}{c} A \\ I \end{array}\right]\otimes \bm{x}'-\left[\begin{array}{c} y \\ \bm{x}(k-1)-\lambda/2 \end{array}\right]\right\|_{2}^2$$
\EndWhile
\State return approximate solution $\bm{x}$

\end{algorithmic}

\end{algorithm}

\begin{example}
Returning to the problem of Example~\ref{invpegg}. We repeat our analysis of the time-series data only this time we include a regularization term when computing each row via \eqref{rowA}. The results of our regularized inference are as follows 
$$
 {A}(\sigma=1,\lambda=10)=\left[\begin{array}{cccc} -\infty &  15 &  -\infty &  -\infty   \\  13.9 &  -\infty &  -\infty &  11  \\  14 &  -\infty &  -\infty &  -\infty   \\  14.9 &  -\infty &  -\infty &  -\infty   \\  \end{array}\right], \quad {A}(\sigma=5,\lambda=10)=\left[\begin{array}{cccc} 8.05 &  14 &  10.5 &  -\infty   \\  13.7 &  -\infty &  9.03 &  10.9  \\  13.8 &  -\infty &  -\infty &  -\infty   \\  14.1 &  8.97 &  7.26 &  7.4  \\  \end{array}\right].
 $$
\end{example}
Note that any entry with little of no evidence is set to minus infinity and that the remaining entries are all fairly accurate approximations of the entries in the true system matrix $M$. Table~\ref{resftab} shows that applying the regularization penalty with $\lambda=10$ only results in a tiny degradation in the solutions fit to the data.

\begin{table}
\centering
\caption{For the inverse problem of Example~\ref{invpegg}. Squared Frobenius norm residual $\|A\otimes X(:,1:n)-X(:,2:n+1)\|_{F}^2$, for low and high noise orbits, with and without regularization penalty. All numeric values given to two decimal places.}\label{resftab}
\begin{tabular}{c|c|c|c}
 & $M$ & $A(\sigma,\lambda=0)$ &  $A(\sigma,\lambda=10)$ \\
\hline
$ \sigma=1$ & 233.78 &  227.41 & 251.86 \\
\hline
$ \sigma=5$ & 5308.58 & 5267.86 & 5275.12
\end{tabular}
\end{table}

\section{Network structure analysis}\label{netset}

In this section we show how min-plus low-rank matrix approximation can be used to analyze a networks structure. Consider the simple tripartite network illustrated in Figure~\ref{smallnetfig} (a) and let $M_{L}\in\Rmin^{n\times d}$, $M_{R}\in\Rmin^{d\times m}$ and $D\in\Rmin^{n\times m}$ be the matrices described in Example~\ref{integg1}. Now suppose that we do not know $M_{L}$ or $M_{R}$ but that we are able to observe $C=D+\zeta$, where $\zeta$ is an $n\times d$ matrix of i.i.d. $(0,\sigma)$ Gaussians. A maximum likelihood estimate for $M_{L}$ and $M_{R}$, i.e. for the edge lengths, is obtained by solving Problem~\ref{minfac}, which is to compute a best fit low rank factorization approximation for $C$.
\begin{problem}\label{minfac}
For $C\in\Rmin^{n\times m}$ and $0<d\leq \min\{n,m\}$, we seek
$$
\min_{A\in\Rmin^{n\times d}, ~ B\in\Rmin^{d\times m}}\|C-A\boxtimes B\|_{F}^{2}.
$$
\end{problem}
For a permutation $\pi\in\Pi_{d}$ and a vector $s\in\R^{d}$, consider the min-plus matrices $Q,P\in\Rmin^{d\times d}$, defined by
\begin{equation}
q_{ij}=\left\{\begin{array}{cc} s_{i}, & \hbox{if $i=\pi(j)$,} \\ \infty, & \hbox{otherwise,}\end{array}\right.\quad p_{ij}=\left\{\begin{array}{cc} -s_{i}, & \hbox{if $i=\pi^{-1}(j)$,} \\ \infty, & \hbox{otherwise.}\end{array}\right.
\end{equation}
Then $Q\boxtimes P=I$, where $I$ is the min-plus identity matrices with zeros on the diagonal and infinities off of the diagonal. Now note that for any $A\in\Rmin^{n\times d}, ~ B\in\Rmin^{d\times m}$ we have 
\begin{equation}
\|C-A\boxtimes B\|_{F}^{2}=\|C-(A\boxtimes P)\boxtimes (Q\boxtimes B)\|_{F}^{2}.
\end{equation}
Hence solutions to Problem~\ref{minfac} can be partitioned into equivalence classes modulo permutation and translation of the columns of $A$ and rows of $B$. We can fix a single solution from each of these equivalence classes by requiring 
\begin{equation}\label{conditionsfac}
\min_{i=1}^{n}a_{ij}=0 ,~\hbox{for $j=1,\dots,d$}, \quad a_{n,1}\geq a_{n,2}\geq\dots\geq a_{n,d}.
\end{equation}
Algorithm~\ref{irsls} takes a simple approach to solving Problem~\ref{minfac} by alternately updating the rows of $A$ and the columns of $B$. Each of these row/col updates requires the solution of an $n\times d$ or $m\times d$ min-plus $2$-norm regression problem.
\begin{algorithm}
\caption{ \label{minpfacalg} {\bf (min-plus approximate factorization)} 
Given an initial guess $A\in\Rmin^{n\times d}, ~ B\in\Rmin^{d\times m}$, returns an approximate solution to Problem~\ref{minfac}.}

\medskip

\begin{algorithmic}[1]

\While{not converged}
\For {$i=1,\dots,n$}
\State $A(i,:)\mapsfrom \arg\min_{\bm{x}\in\Rmin^{1\times d}}\|\bm{x}\boxtimes B-C(i,:)\|_{2}^{F}$
\EndFor
\For {$j=1,\dots,m$}
\State $B(:,j)\mapsfrom\arg\min_{\bm{x}\in\Rmin^{d}}\|A\boxtimes \bm{x}-C(:,j)\|_{2}^{F}$
\EndFor
\EndWhile
\State translate and permute to satisfy \eqref{conditionsfac}.
\State return factors $A,B$
\end{algorithmic}

\end{algorithm}

\begin{example}
For the true network factors $M_{L}$ and $M_{R}$ below we generate the matrix $C=M_{L}\boxtimes M_{R}+\zeta$, where $\zeta$ is an $n\times d$ matrix of i.i.d. $(0,1)$ Gaussians. 
$$
M_{L}=\left[\begin{array}{cc} 1 &  1  \\  0 &  4  \\  7 &  2  \\  5 &  0  \\  3 &  1  \\  \end{array}\right],\quad  M_{R}=\left[\begin{array}{cc} 3 &  8  \\  4 &  8  \\  11 &  12  \\  10 &  9  \\  3 &  13  \\  \end{array}\right], \quad C=\left[\begin{array}{ccccc} 3.59 &  6.07 &  12.5 &  10.2 &  3.57  \\  3.42 &  2.75 &  10.8 &  11 &  3.21  \\  11.8 &  10.3 &  15.4 &  9.74 &  10.6  \\  5.91 &  8.62 &  11.9 &  9.7 &  9.77  \\  3.98 &  8.04 &  14.5 &  10.2 &  6.39  \\  \end{array}\right].
$$
We apply Algorithm~\ref{minpfacalg} to $C$ and obtain the maximum likelihood estimates $A$ and $B$ for the factors as follows
$$
A=\left[\begin{array}{cc} 1.14 &  0.769  \\  0 &  0.942  \\  7.54 &  2.16  \\  4.8 &  0  \\  2.8 &  1.36  \\  \end{array}\right], \quad B=\left[\begin{array}{cc} 2.57 &  7.96  \\  4.32 &  8.86  \\  11.2 &  13.4  \\  11 &  9.36  \\  3.54 &  9.01  \\  \end{array}\right].
$$
For these matrices we have $\|C-M_{L}\boxtimes M_{R}\|_{F}^2=27.77$ and $\|C-A\boxtimes B\|_{F}=22.09$.
\end{example}

Now consider the undirected bipartite graph of Figure~\ref{smallnetfig} (b) and let $M\in\Rmin^{n\times d}$ be the matrix described in Example~\ref{integg2}. Then $D=M\boxtimes M^{\top}$ is the $n\times n$ min-plus matrix such that $d_{ij}$ is equal to the length of the shortest two edge path from $x(i)$ to $x(j)$. Now suppose that we do not know $M$ but that we are able to observe $C=D+\zeta$, where $\zeta$ is an $n\times d$ matrix of i.i.d. $(0,\sigma)$ Gaussians. A maximum likelihood estimate for $M$, i.e. for the edge lengths, is obtained by solving Problem~\ref{minfacsym}. 
\begin{problem}\label{minfacsym}
For $C\in\Rmin^{n\times n}$ and $0<d\leq n$, we seek
$$
\min_{A\in\Rmin^{n\times d}}\|C-A\boxtimes A^{\top}\|_{F}^{2}.
$$
\end{problem}
Alternately we may want to allow paths of length zero between a vertex and itself so that $d_{ii}=0$ for $i=1,\dots,n$. In this case we have $D=I\boxplus M\boxtimes M^{\top}$, where $I$ is the min-plus identity matrix with zeros on the diagonal and plus infinities off of the diagonal. If we observe $C=D+\zeta$ as before then the maximum likelihood estimate for the edge lengths can be obtained by solving Problem~\ref{minfacsym2}. 
\begin{problem}\label{minfacsym2}
For $C\in\Rmin^{n\times n}$ and $0<d\leq n$, we seek
$$
\min_{A\in\Rmin^{n\times d}}\sum_{i\neq j}\big(C-A\boxtimes A^{\top}\big)_{ij}^{2}.
$$
\end{problem}
We are unable to adapt Algorithm~\ref{minpfacalg} to solve Problem~\ref{minfacsym} or \ref{minfacsym2} as we do not have any compatible way of enforcing symmetry in the factors at each step. Instead we apply Newton's method, using the whole of the approximate factor as the iterate. See Algorithm~\ref{newtomsymfac1}.

\begin{example}
For the true network factors $M$ below we compute the matrix $C=M\boxtimes M^{\top}\boxplus I+\zeta$, where $\zeta\in\R^{5\times 5}$ has zeros on the diagonal and i.i.d. $(0,1)$ Gaussians off of the diagonal. We apply Algorithm~\ref{newtomsymfac1} to $C$ and obtain the maximum likelihood estimate $A$ for the factor as follows
$$
M=\left[\begin{array}{cc} 8 &  4  \\  8 &  3  \\  1 &  8  \\  2 &  7  \\  7 &  7  \\  \end{array}\right]
, \quad C=\left[\begin{array}{ccccc} 0 &  7.53 &  9.87 &  11 &  11  \\  7.93 &  0 &  9.03 &  10.6 &  10.2  \\  9.12 &  9.75 &  0 &  3.66 &  8.86  \\  10.6 &  10.3 &  3.44 &  0 &  9.07  \\  11.5 &  10.2 &  8.07 &  9.48 &  0  \\  \end{array}\right], \quad A=\left[\begin{array}{cc} 8.58 &  4.68  \\  8.4 &  2.76  \\  1.03 &  9.75  \\  2.19 &  8.62  \\  7.3 &  7.16  \\  \end{array}\right].
$$
For these matrices we have 
$$
\sum_{i\neq j}\big(C-M\boxtimes M^{\top}\big)_{ij}^{2}=6.19, \quad \sum_{i\neq j}\big(C-A\boxtimes A^{\top}\big)_{ij}^{2}=2.73.
$$
\end{example}

We have shown how min-plus matrix factorization is able to approximately recover the edge lengths from noisy observations of shortest path distances for tripartite and bipartite networks of the sorts illustrated in Figure~\ref{smallnetfig}. Now suppose that $G$ is an undirected network of unconstrained structure, with vertices $x(1),\dots,x(n)$ and that $D\in\Rmin^{n\times n}$ records the shortest path distances between the vertices of $G$. In this context a solution to Problem~\ref{minfacsym2} provides a set of $d$ additional hub vertices $y(1),\dots,y(d)$, such that the distance from $x(i)$ to $y(k)$ is given by $a_{ik}$. These hubs then approximate the shortest path distances through the original network by 
$$
d_{ij}\approx \min_{k=1}^{d}a_{ik}+a_{kj}, \quad \hbox{for $i\neq j$}.
$$
Thus the symmetric low rank matrix factorization approximation $D\approx A\boxtimes A^{\top}\boxplus I$ captures the distances of the original network with a bipartite graph structure containing fewer connections. In this sense, the approximation can be viewed as a min-plus linear model order reduction of the network. Such a reduction may be useful as a means to characterize a network's structure and as a means to extract a small number of features that can be used to describe the location of the vertices in a network. Note that for larger networks, such a factorization could be obtained from only a small subset of the rows of $D$, so that computing all of the shortest path distances is not necessary for this reduction. 

\begin{example}\label{dolphinegg}

We use the dolphin social network presented in \cite{dolphins}. This network consists of 62 vertices, each of which represents a different dolphin, with an edge connecting two dolphins if they are observed to regularly interact. This small social network is frequently used to test or illustrate data analysis techniques.  We first compute the distance matrix $D\in\Rmin^{62\times 62}$, with $d_{ij}=$ the length of the shortest path through the network from dolphin $i$ to $j$. Next we apply Algorithm~\ref{newtomsymfac1} to compute the best fit rank-3 factor $A\in\Rmin^{62\times 3}$.

We can think of the columns of $A$ as representing neighborhoods in the network. If $a_{ik}$ is small then dolphin $i$ is close to neighborhood $k$, and therefore dolphin $i$ will be close to any other dolphin that is also close to neighborhood $k$. Otherwise if $a_{ik}$ is large then dolphin $i$ is far from neighborhood $k$ and will not be close to any dolphin that is close to neighborhood $k$, unless they share some other mutually close neighborhood. 

Just as in conventional principal component analysis, the rows of $A$ can be thought of as latent factors that parametrize the rows of $D$. Equivalently, the $i$th row $A_{i\cdot}$ encapsulates information about dolphin $i$'s position in the network, so we can study the structure of the network by examining $\{A_{i\cdot}~:~i=1,\dots,n\}$, which is simply a scattering of points in $\R^3$. Figure~\ref{dolphins1} displays the dolphin social graph as well as the rows of $A$. Note that we have plotted the reciprocals of the entries in $A$, so that a large value of $1/a_{ik}$ indicates that dolphin $i$ is close to neighborhood $k$. The dolphins have then been color coded according to their closest neighborhood. Comparing the network to the scattering of points, it is clear that the min-plus factorization has captured the predominant structure of the graph. We can easily identify strongly connected groups or clusters of dolphins and spot individuals that provide bridges between different groups.

\begin{figure}
\begin{center}
\subfigure[Dolphin social network.]{\includegraphics[scale=0.45]{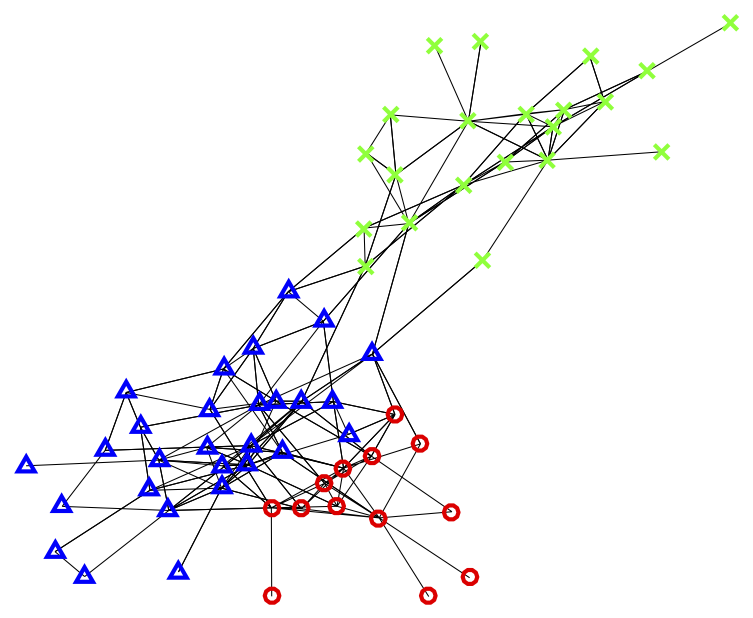}} \hspace{15mm}  \subfigure[Min-plus latent factors.]{\includegraphics[scale=0.5]{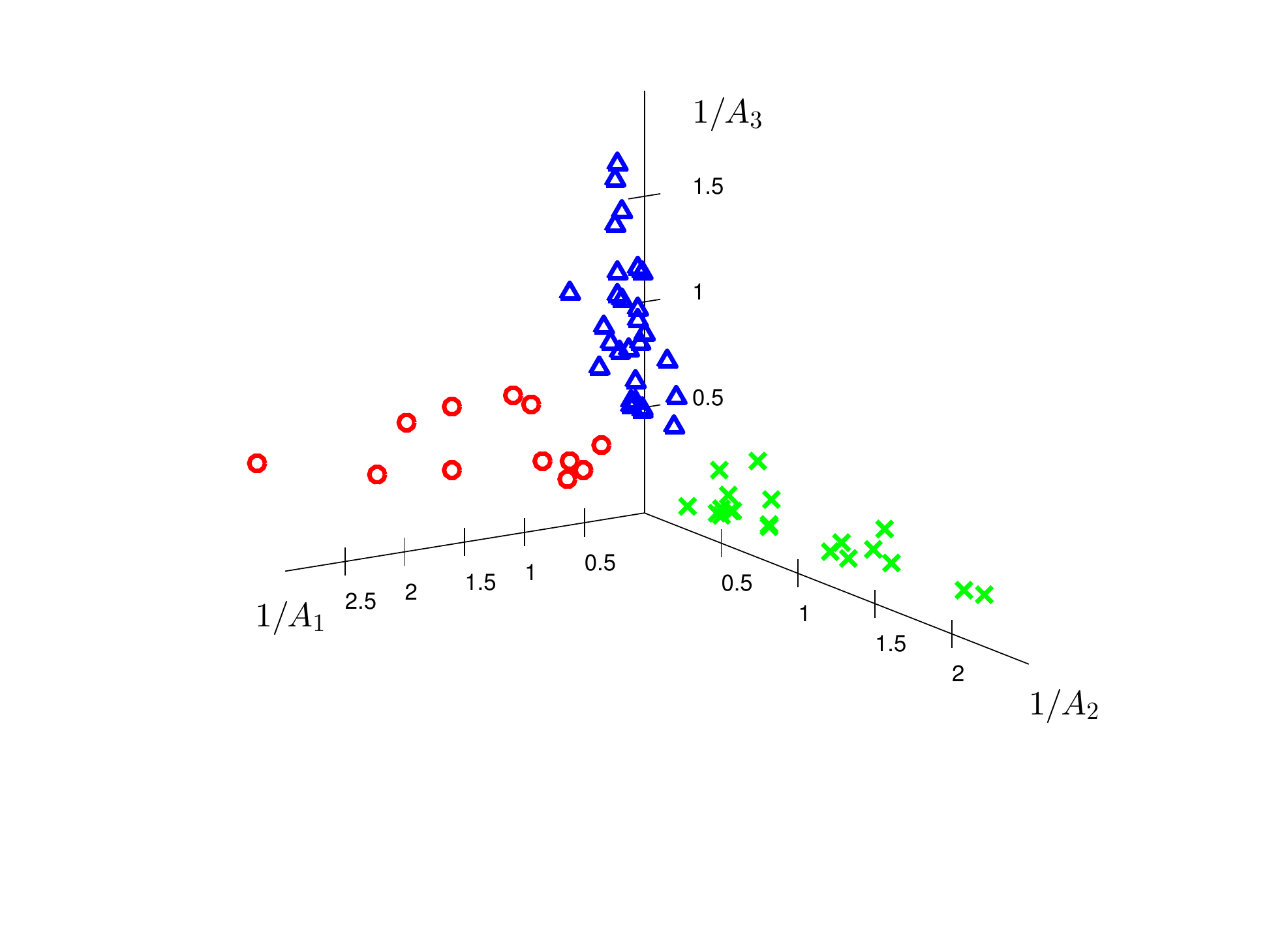}}

\caption{For the network of Example~\ref{dolphinegg}.}\label{dolphins1}
\end{center}
\end{figure}

\end{example}
\section{Polynomial regression}\label{polf}

A univariate, degree-d, max-plus polynomial is a function $p:\Rmax\mapsto\Rmax$ of the form 
$$
p_{\bm{a}}(x)=\bigoplus_{n=0}^{d}a_{n}\otimes x^{\otimes n}=\max_{n=0}^{d}(a_{n}+nx),
$$
where $\bm{a}\in\Rmax^{d+1}$ are the polynomial coefficients. More generally a multivariate max-plus polynomial is a function $p:\Rmax^m\mapsto\Rmax$ of the form 
$$
p_{\bm{a},S}(\bm{x})=\bigoplus_{n=1}^k \bm{a}_{n}\bigotimes_{i=1}^{d} x_{i}^{\otimes s_{ni}}=\max_{n=1}^{k}\big(a_{n}+S(n,:)\bm{x}\big),
$$
where $\bm{a}\in\Rmax^{k}$ are the polynomial coefficients and $S\in\mathbb{N}^{k \times d}$ is the matrix of monomial slopes. 

\begin{problem}\label{polregprob}
For $\bm{x}(1),\dots,\bm{x}(N)\in\Rmax^m$, $\bm{y}\in\Rmax^N$ and $S\in\mathbb{N}^{k \times d}$, we seek
$$
\min_{\bm{a}\in\Rmax^k}\sum_{i=1}^{N}\Big(p_{\bm{a},S}\big(\bm{x}(i)\big)-y_i\Big)^2.
$$
\end{problem}
Just as in the classical case it is straightforwards to convert a max-plus polynomial regression problem into a max-plus linear regression problem. Let $X\in\Rmax^{N\times k}$ be the matrix with $(X)_{ij}=S(j,:)\bm{x}(i)$, then the solution to Problem~\ref{polregprob} is given by
\begin{equation}\label{polylin}
\min_{\bm{a}\in\Rmax^k}\|X\otimes \bm{a}-\bm{y}\|_{2}^2,
\end{equation}
Note that this result extends to the more general case $S\in\R^{k\times d}$, which corresponds to fitting a convex piecewise affine function with fixed slope values.
\begin{example}\label{pol1}
We sample $20$ i.i.d $(0,1)$ Gaussian data points $x_{1},\dots,x_{20}\in\Rmax$. Next we compute $y_{i}=p_{\bm{a}}(x_{i})+\zeta_{i}$, where $p_{\bm{a}}$ is the max-plus polynomial with $\bm{a}=[0,1,0]^{\top}$ and $\zeta_{1},\dots,\zeta_{20}$ are i.i.d $(0,0.5)$ Gaussians. We solve Problem~\ref{polregprob} by converting it into a max-plus linear $2$-norm regression problem as in \eqref{polylin} and obtain the maximum likelihood estimate for the polynomial coefficients $\bm{b}=[ 0.19,1.08,-0.12]$. Figure~\ref{polypolt} (a) is a plot of the data $(x_{i},y_{i})_{i=1}^{20}$ along with the graph of the function $p_{\bm{b}}$. For this data we have $\|X\otimes \bm{a}-y\|_{2}=1.8247$ and $\|X\otimes \bm{b}-\bm{y}\|_{2}=1.7041$.

\end{example}

\begin{example}\label{pol2}
We sample $200$ i.i.d data points $\bm{x}(1),\dots,\bm{x}(200)\in\Rmax^2$ uniformly from $[-1,1]^{2}$. Next we set $\bm{y}\in\Rmax^{200}$, with $y_{i}=\|\bm{x}(i)\|_2^2$, for $i=1,\dots,200$. We choose slope values
$$
S=\left[\begin{array}{ccccc} 0 & 1 & -1 & 0 & 0  \\ 0 & 0 & 0 & 1 & -1 \end{array}\right]^{\top}, 
$$
then solve Problem~\ref{polregprob} by converting it into a linear regression problem as in \eqref{polylin} and obtain the maximum likelihood estimate for the polynomial coefficients $\bm{a}=[0.10,   -0.63,   -0.67,   -0.77,   -0.75]$. Figure~\ref{polypolt} (b) is a plot of the data $(\bm{x}(i),y_{i})_{i=1}^{200}$ along with a surface plot of the function $p_{\bm{a},S}$.

\end{example}

\begin{figure}[t]
\begin{center}
\subfigure[For the problem of Example~\ref{pol1}]{\includegraphics[scale=0.4]{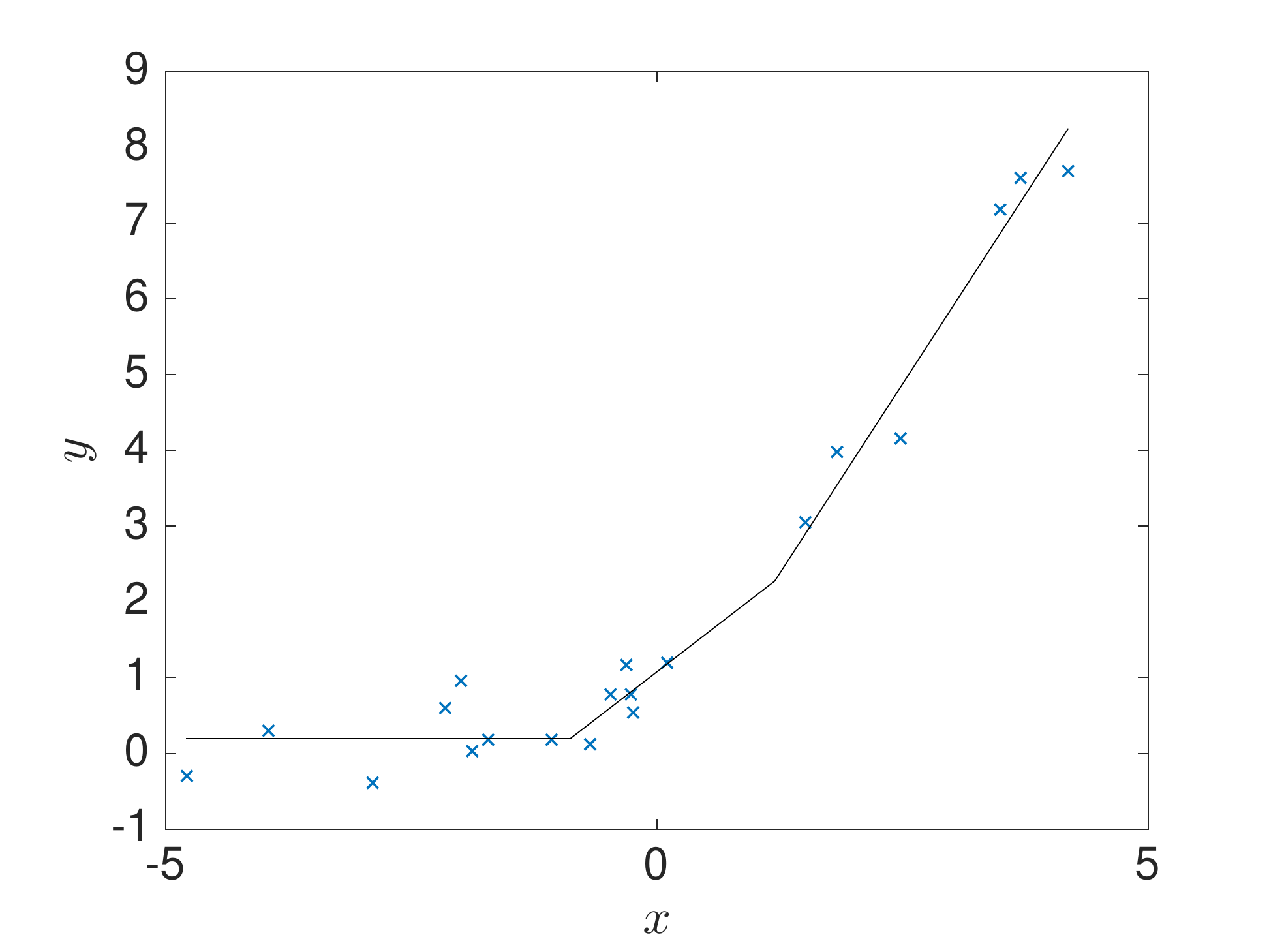}}\subfigure[For the problem of Example~\ref{pol2}. Red diamonds/black crosses are those datapoints over/under-approxiamted by the fitted function.]{\includegraphics[scale=0.45]{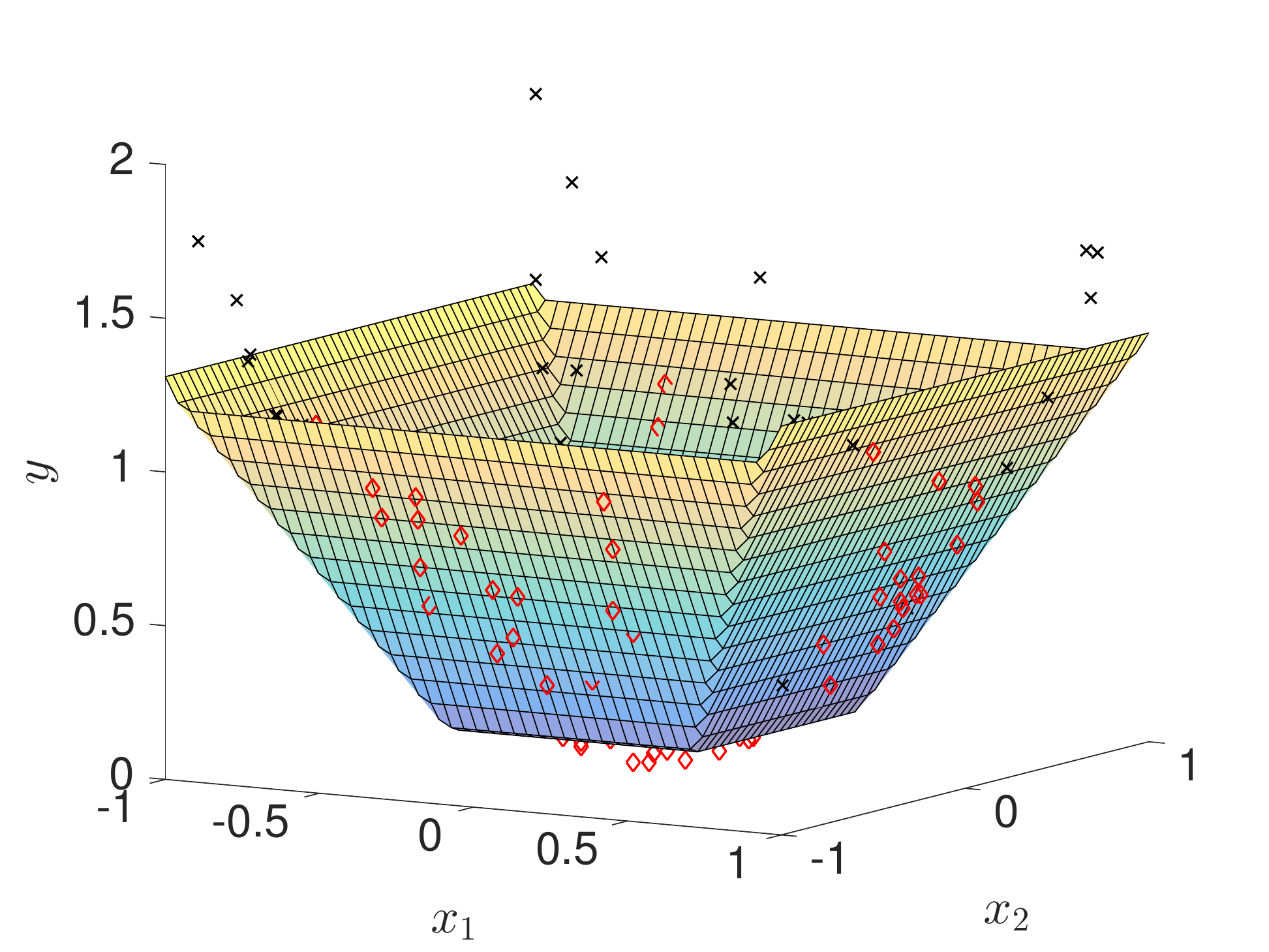}} 
\caption{Polynomial least squares curve fitting.}\label{polypolt}
\end{center}
\end{figure}

\section*{Discussion}

In this paper we presented theory and algorithms for max-plus $2$-norm regression and then demonstrated how they could be applied to 
three different inverse problems. Namely inferring a max-plus linear dynamical systems model from a noisy time series recording, 
inferring the edge lengths of a network from shortest path information and fitting a max-plus polynomial to data. This work leaves open several direction for future research, which can be grouped into four topics; theory, algorithms, inverse problems and applications. 

We saw that whilst the max-plus $\infty$-norm regression problem could be solved easily the $2$-norm variant was more difficult. Clearly the $\infty$-norm is readily compatible with max-plus algebra, whereas the $2$-norm is more suited to classical linear algebra. The max-plus $2$-norm regression problem forces us to bring these different worlds together. Our approach in Section~\ref{2normo} was to develop a formulation of the max-plus matrix vector multiplication map $x\mapsto A\otimes x$, its domain and image etc... in terms of classical linear algebra. With this formulation we could readily apply classical optimization techniques such as steepest descent and Newton's method. However, whilst the matrix vector multiplication map is an extremely simple object viewed through max-plus algebra it become complex and unwieldily when formulated in classical terms. Developing new ways to formulate this problem could lead to the development of superior algorithms. 

Whilst Algorithm~\ref{newton} cannot be guaranteed even to return a local minimum, we did find that it worked well enough in practice. However, developing efficient algorithms that are able to provide some better performance guarantees would be very desirable. Theorem~\ref{setsetset} seems to stand somewhat in the way of this goal, but note that the theorem relates to a specific point for a highly structured (i.e. degenerate) problem. So it may still be possible to developing an efficient residual descending algorithm. 

As noted in the introduction, most applications of max-plus linear dynamical systems use petri-net models, which result in highly structured iteration matrices and noise processes. Further work is needed to adapt the approach used in Section~\ref{tss} to this setting. Similarly to include a control input as in \cite{1184996}. Framing these more general inverse problems explicitly in terms of linear regression problems might inspire new techniques, possibly by trying to develop max-plus analogues of classical linear systems theory. For example, if the matrix $X$ containing the time series vectors can be well approximated by a max-plus low-rank matrix product then what does this tell us about the system? Similarly the techniques we outlined in Section~\ref{netset} only covered a tiny faction of the possible min-plus network  inference problems. For instance we do not yet have a method for the inverse problem associated with the network in Example~\ref{integg2}.  

 We demonstrated in Example~\ref{dolphinegg} that our min-plus linear model order reduction techniques could be applied to analyze `real-world' data, highlighting neighborhood structure in a social network. More work is needed to explore the possible application of this approach. It is also noted in \cite{ANGULO20171} that max-plus low-rank approximate matrix factorization could have applications in non-linear image processing, which provides additional motivation for developing these techniques further.

\section*{Acknowledgement}This work was supported by a University of Bath, Institute for Mathematical Innovation, 50th Anniversary Prize Fellowship. We also thank Henning Makholm for answering a question on Math Stack Exchange, which helped in the formulation of Theorem~\ref{setsetset}.

\bibliographystyle{is-abbrv}
\bibliography{trop}

\newpage

\section{Appendix}
\subsubsection{Brute force method}

A simple way to solve the max-plus $2$-norm regression problem exactly is to search through all of the feasible patterns, computing the normal projections and checking their admissibility for each one in turn. We  then select the closest admissible normal projection for our solution.

To search efficiently through the set of all feasible patterns we consider the tree $T$, with vertices $V_{0},\dots, V_{n}$ at depths $0,1,\dots,n$ respectively. The depth $k$ vertices $V_{k}$ are ordered $k$-tuples of the form $v=(P_{1},\dots,P_{k})$, with $P_{i}\subset\{1,\dots,d\}$, for $i=1,\dots,k$. A vertex $v\in V_{k}$ is parent to  $v'\in V_{k+1}$ if and only if $v'=(v,P_{k+1})$ for some $P_{k+1}\subset\{1,\dots,d\}$. 

In analogy to the feasibility matrix for a pattern \eqref{feasibilitymatrixdef}, define the feasibility matrix for a vertex $v\in V_{k}$, by $F_{v}\in\Rmax^{d\times d}$, with 
\begin{equation}
f_{jk}=\left\{\begin{array}{cc} 0, & \hbox{for $j=k$,} \\ \max\big\{-\infty,\max \{a_{ik}-a_{ij}~:~ j\in P_{i}, ~ i\leq k\}\big\},   & \hbox{otherwise.}\end{array}\right.
\end{equation}
We say that $v\in V_{k}$ is feasible if $\lambda(F_{v}^{\star})=0$. Note that the set of feasible leaf vertices is identical to the set of feasible patterns of support. It is easy to show that if $v\in V_{k}$ is feasible then all of its ancestors are feasible and at least one of its children is feasible. Also if $v\in V_{k}$ is not feasible then none of its children are feasible. 

Next we define an order $\trianglelefteq_{L}$ on the vertices of $T$ by taking an arbitrary ordering $\trianglelefteq$ on the subsets of $\{1,\dots,d\}$ and extending this order lexicographically to $T$. We start at the vertex $v_{0}=()$ and proceed to search through $T$, in order of $\trianglelefteq_{L}$. At each vertex $v\in V_{k}$, we check for feasibility by computing $\lambda(F_{v}^{\star})$, with worst case cost $\uc{O}(d^3)$. If $\lambda(F_{v}^{\star})>0$, then $v$ is non-feasible and we skip all of its decedents. Whenever we reach a feasible leaf vertex $P\in V_{n}$, we compute $\Phi(P,\bm{y})$ and $\Psi(P,\bm{y},\underline{-\infty})$ with cost $\uc{O}(n)$. We check for admissibility by computing $F_{P}\otimes \Psi(P,\bm{y},\underline{-\infty})$ with cost $\uc{O}(d^2)$. If $F_{P}^{}\otimes \Psi(P,\bm{y},\underline{-\infty})=\Psi(P,\bm{y},\underline{-\infty})$ then we compute $\|\bm{y}-\Phi(P,\bm{y})\|_{2}$ and if this residual is the best that we have seen so far, then we save $\Psi(P,\bm{y},\underline{-\infty})$ as the interim optimal solution. The algorithm terminates when we have either checked or skipped all vertices in $T$. See Algorithm~\ref{brute}.

\begin{algorithm}
\caption{ \label{brute}
Returns an optimal solution to Problem~\ref{mppnreg} with $p=2$.}

\medskip

\begin{algorithmic}[1]

\State set $v=()$
\State set $r_{\min}=\infty$
\While{not all of $T$ explored or skipped}
\If{$\lambda(F_{v}^{\star})=0$ $\Leftrightarrow$ $v$ is feasible}
\If {$v=P$ is a leaf vertex}
\State compute $\Phi(P,\bm{y})$ and $\Psi(P,\bm{y},\underline{-\infty})$
\If {$F_{P}^{}\otimes \Psi(P,\bm{y},\underline{-\infty})=\Psi(P,\bm{y},\underline{-\infty})$  $\Leftrightarrow$ $\Phi(P,\bm{y})$ is admissible}
\If {$\|\Phi(P,\bm{y})-\bm{y}\|_{2}<r_{\min}$}
\State update $r_{\min}\mapsfrom \|\Phi(P,\bm{y})-\bm{y}\|_{2}$ and $x\bm{x}\mapsfrom \Psi(P,\bm{y},\underline{-\infty})$
\EndIf
\EndIf
\EndIf
\Else
\State  skip decedents of $v$
\EndIf
\State go to next $v$
\EndWhile
\State return solution $\bm{x}$

\end{algorithmic}

\end{algorithm}

Applied to an $n\times d$ problem, Algorithm~\ref{brute} must check the feasibility of $c_{v}$ vertices and compute normal projections and check admissibility for $c_{l}$ leaf vertices. From \eqref{numer} we have 
$$
c_{l}\leq\sum_{k=1}^{n} \frac{(n+d-k-1)!}{(n-k)!\cdot (d-k)! \cdot (k-1)!},
$$
and a rough bound for $c_{v}$ is given by $c_{v}\leq nd c_{l}$. Computing the normal projection has cost $\uc{O}(n)$ and checking feasibility, using a max-plus eigenvalue solver, has worst case cost $\uc{O}(d^3)$. Thus the worst case cost of Algotithm~\ref{brute} is $\uc{O}(c_{v}d^3)$.

Another, perhaps more efficient, exhaustive approach is to confine the search to patterns $P\in\uc{P}\big(\{1,\dots,d\}\big)^n$, such that $|P_{i}|=1$, for $i=1,\dots,n$. This drastically reduces the size of the tree to be searched through, but we must then solve a more difficult local problem on each admissible leaf vertex. Instead of having to compute the normal projection, which is an unconstrained quadratic program, we need to compute 
\begin{equation}
\min_{\bm{x}\in \Cl\big(X(P)\big)}\|A_{P}(\bm{x})-\bm{y}\|_{2}^2,
\end{equation}
which is a linearly constrained quadratic programming problem. Although such exhaustive search algorithms will never be suitable for applying to very large problems, they can still be extremely valuable for use on smaller data sets and as a way to benchmark the performance of faster approximate algorithms.

\subsection{NP-hardness of finding descent directions}

For $A\in\Rmax^{n\times d}$, $\bm{y}\in\Rmax^n$ and $\bm{x}\in\Rmax^d$ we say that $\bm{z}\in\R^d$ is an \emph{descent direction} for $\bm{x}$ if there exists $\epsilon>0$ such that
\begin{equation}
R(\bm{x}+\mu \bm{z})<R(\bm{x}),
\end{equation}
for all $0<\mu\leq \epsilon$.

\begin{lemma}\label{pointdowny}
Let $A\in\{0,-\infty\}^{n\times d}$ and $\bm{y}\in\Rmax^n$. Then $\bm{z}\in\R^d$ is a descent direction for $\underline{0}\in\Rmax^d$ if and only if
\begin{equation}\label{dotpro}
\langle A\otimes \bm{z},\bm{y}\rangle<0.
\end{equation}
Moreover, in the case that $\sum_{i=1}^{n}y_{i}=0$, if $\underline{0}\in\Rmax^d$ has a descent direction $\bm{z}\in\R^d$ then is has a descent direction $\bm{z}'\in\{0,1\}^{d}$.
\end{lemma} 
\begin{proof}
The first part follows by taking the derivative in the definition of a descent direction. For the second part note that the zero sum condition means that $\bm{z}$ must have some strictly positive and negative components. We can therefore apply an affine map component wise to $\bm{z}$ to set its largest entry to $+1$ and its smallest entry to $-1$, without loosing the property that $\bm{z}$ is a descent direction. Now list the components of $\bm{z}$ in ascending order $z_{j_{1}}=-1\leq z_{j_{2}} \leq \dots \leq z_{j_{d}}=1$. Define an equivalence relation $\sim$ on the components of $\bm{z}$ that identifies components that are equal to each other. For each equivalence class $S$ compute the sign of 
$$
\sum_{j\in S}\sum_{i~:~ \ell(i)=j}y_{i},
$$
where $\ell\in \{1,\dots,d\}^{n}$ is the subparttern of $P(\bm{z})$.  Whenever there are two consecutive equivalence classes, not including the first or last class, with signs $+$ and $-$ respectively, merge them together and position them at the midpoint of the two original classes. If there is ever a class with zero sum and therefore no sign, merge it with an adjacent class. Whenever the second class in the order has a $-$ sign merge it with the first class, whenever the second from last class has a $+$ sign merge it with the last class. Continue in this way until there are only two classes. By construction if we set $z_{j}=1$ for all $j$ in the upper class and $z_{j}=-1$ for all $j$ in the lower class then we have a descent direction and therefore from the zero sum property we must also have that choosing $\bm{z}'\in\{0,1\}^d$, with  $z'_{j}=1$ for all $j$ in the upper class and $z'_{j}=0$ for all $j$ in the lower class also gives a  descent direction.
\end{proof}
\begin{theorem}\label{setsetset}
Let $F=\{F_{i}~:~i=1,\dots,m\}$ be a family of subsets of $\{1,\dots,n\}$ with $\cup_{i=1}^{m} F_{i}=\{1,\dots,n\}$ and let $1< k< m$, then there exists a subset $\{j_{1},\dots,j_{k}\}\subset\{1,\dots,m\}$ such that $\cup_{i=1}^{k} F_{j(i)}=\{1,\dots,n\}$, if and only if $\underline{0}\in\Rmax^m$ has a descent direction for the regression problem with $A\in\{0,-\infty\}^{(n+m+\frac{m(m-1)}{2}+1)\times m}$, given by
$$
a_{ij}=\left\{\begin{array}{cl} 0 & \hbox{if $i\in F_{j}$, for $i=1,\dots,n$,}  \\ 0 & \hbox{if $(i-n)=j$, for $i=n+1,\dots,n+m$,}  \\  0 & \hbox{if $j\in p(i-n-m)$, for $i=n+m+1,\dots,n+m+m(m-1)/2$} \\  0 & \hbox{if $i=n+m+m(m-1)/2+1$} \\ -\infty & \hbox{otherwise,}\end{array}\right.
$$  
where $p(1),\dots,p\big(m(m-1)/2\big)$ is a list of all unordered pairs of elements of $\{1,\dots,m\}$, and with $\bm{y}\in\R^{(n+m+\frac{m(m-1)}{2}+1)}$, given by
$$
y_{i}=\left\{\begin{array}{cl} c & \hbox{for $i=1,\dots,n$,} \\ a & \hbox{for $i=n+1,\dots,n+m$,} \\ b  & \hbox{for $i=n+m+1,\dots,n+m+\frac{m(m-1)}{2}$,} \\ -nc-ma-\frac{m(m-1)}{2}b & \hbox{for $i=n+m+\frac{m(m-1)}{2}+1$,}\end{array}\right.
$$
were the parameters are given by 
$$
a=-\frac{1}{2}(m-k-1), \quad b=1, \quad c=-(|a|+|b|)m^2.
$$
\end{theorem}
\begin{proof} First note that from Lemma~\ref{pointdowny}, if  $\underline{0}\in\Rmax^m$ has a descent direction, then it has a descent direction $\bm{z}\in\{0,1\}^n$. Let $J=\{j~:~z_{j}=1\}$. Then $J$ must be non-empty so \eqref{dotpro} must contain the term coming from the $i=n+m+\frac{m(m-1)}{2}+1$ entry of $\bm{y}$, and therefore, becuase $c$ is chosen to be much larger than $a$ and $b$, we must have $\cup_{j\in J}F_{j}=\{1,\dots,n\}$. Now suppose that $J$ contains precisely $p$ entries, then
$$
\langle A\otimes \bm{z},\bm{y} \rangle=\frac{-1}{2}(p-m)\big(p-(k+\frac{1}{2})\big)\left\{\begin{array}{cc} =0 & \hbox{if $p=m$} \\  >0 & \hbox{if $k<p<m$} \\  <0  & \hbox{if $p\leq k$}\end{array}\right.
$$
Therefore if $\bm{z}\in\{0,1\}^n$ is a descent direction then we have $|J|\leq k$ and $\cup_{j\in J}F_{j}=\{1,\dots,n\}$.
\end{proof}

\subsection{Symmetric low-rank approximate min-plus factorization algorithm}

For $D\in\Rmin^{n\times n}$ the squared residual $R~:~\Rmin^{n\times d}\mapsto \R$, defined by $R(A)=\|D-A\otimes A^{T}\|_{F}^{2}$,
is piecewise quadratic, continuous but non-differentiable. For $A\in\Rmin^{n\times d}$, define $K(A)\in\{1,\dots,d\}^{n\times n}$ by
$$
K(A)_{ij}=\min\big(\arg\min_{k=1}^{d}(a_{ik}+a_{jk})\big).
$$
Then we have
\begin{equation}\label{sfresidual3}
R(A)=R_{K(A)}(A),
\end{equation}
where $R_{K(A)}:\Rmin^{n\times d}\mapsto \R$, is defined by
\begin{equation}\label{sfresidual2}
R_{K(A)}(A')=\sum_{ij=1}^n\big(a'_{i{K}(A)_{ij}}+a'_{j{K}(A)_{ij}}-d_{ij}\big)^2.
\end{equation}
As in the case of the regression problem, Newton's method finds the minimum to the local quadratic piece
\begin{equation}\label{nretonsym}
\uc{N}(A)=\arg\min_{A'\in\Rmin^{n\times d}}R_{K(A)}(A').
\end{equation}
Define $\uc{J}_{A}:\Rmin^{n\times d}\mapsto\Rmin^{n\times d}$, by
\begin{equation}\label{jacobi}
\uc{J}_{A}(A')_{ik}=\frac{d_{ii}{\bf 1}_{ik}+\sum_{j\neq i : {K}(A)_{ij}=k}d_{ij}-a'_{jk}}{2{\bf 1}_{ik}+\sum_{j\neq i : {K}(A)_{ij}=k}1},
\end{equation}
where ${\bf 1}_{ik}=1$ if ${K}(A)_{ij}=k$ and ${\bf 1}_{ik}=0$ otherwise. The map $\uc{J}_{A}$ is the result of applying one iteration of Jacobi's method to the normal equations associated to the linear least squares formulation of \eqref{nretonsym}.

Therefore we can compute Newton's method updates iteratively using $\uc{J}$. However, as in the case of the regression problem, Newton's method is not guaranteed to converge to a local minimum so we propose using approximate Newton updates with undershooting as in Algorithm~\ref{newton}. By using a small fixed number of $\uc{J}$ iterations we can cheaply approximate the Newton step. We then update by moving to a point somewhere between the previous state and the result of our approximate Newton computation. By gradually reducing the length of the step we can avoid getting stuck in the periodic orbits that prevent standard Newton's method from converging. As in the non-symmetric case the choice of initial factorization is important. One possibility is to take a random selection of the columns of $D$. See Algorithm~\ref{newtomsymfac1}. 

\begin{algorithm}
\caption{ \label{newtomsymfac1}
Returns an approximate solution to Problem~\ref{minfacsym}. Parameters are the number of Jacobi iterations per step $t\in\mathbb{N}$ and the shooting factor $\mu\in\mathbb{R}_{+}$. These parameters may be allowed to vary during the computation. }
\begin{algorithmic}[1]
\State draw uniformly at random $\{w_{1},\dots,w_{m}\}\subset\{1,\dots,n\}$
\State set $A=D_{W}$
\While{ not converged}
\State compute $\hat{\uc{N}}(A)=\uc{J}^{~t}_{A}(A)$
\State update $F\mapsfrom \mu\hat{\uc{N}}(A)+(1-\mu)A$
\EndWhile
\State return approximate solution $A$
\end{algorithmic}
\end{algorithm}

Formulating the map $\uc{J}_{A}$ has cost $\uc{O}(n^2d)$ and applying it has cost $\uc{O}(n^2)$. Thus the approximate Newton computation (line 4) has cost $\uc{O}\big(n^2(d+t)\big)$, where $t$ is the number of Jacobi iterations used at each step. Adapting Algorithm~\ref{newtomsymfac1} to solve Problem~\ref{minfacsym2} is simple. We simply ignore any contribution to the residual $R$ or the local residuals $R_{K}$ coming from the diagonal entries of the matrix.

\end{document}